\documentclass[11pt,a4paper,reqno]{amsart}
\usepackage[margin=1in]{geometry}
\usepackage{xcolor}
\usepackage{amsmath}
\usepackage{amsaddr,amsthm,amssymb}
\usepackage{mathrsfs,mathtools}
\usepackage{graphicx}
\usepackage{subcaption}
\usepackage{float}
\usepackage{amsthm}
\usepackage{verbatim}
\usepackage{booktabs}
\usepackage[titletoc]{appendix}
\usepackage{etoolbox}
\usepackage{setspace}
\usepackage{threeparttable}
\usepackage{thmtools}
\usepackage{thm-restate}
\usepackage[numbers]{natbib}

\newtheorem{theorem}{Theorem}
\newtheorem{lemma}{Lemma}
\newtheorem{corollary}{Corollary}
\theoremstyle{definition}

\newtheorem*{assume}{Assumption}

\theoremstyle{remark}
\newtheorem*{remark}{Remark}

\numberwithin{equation}{section}

\makeatletter
\def\blfootnote{\gdef\@thefnmark{}\@footnotetext}
\makeatother

\DeclareMathOperator*{\argmin}{arg\,min}

\makeatletter
\newcommand{\superimpose}[2]{%
  {\ooalign{$#1\@firstoftwo#2$\cr\hfil$#1\@secondoftwo#2$\hfil\cr}}}
\makeatother
\newcommand{\vw}{\mathpalette\superimpose{{\vee}{\wedge}}}

\allowdisplaybreaks

\makeatletter
\renewcommand{\tocsection}[3]{%
	\indentlabel{\@ifnotempty{#2}{\bfseries\ignorespaces#1 #2\quad}}\bfseries#3}
\renewcommand{\tocsubsection}[3]{%
	\indentlabel{\@ifnotempty{#2}{\ignorespaces#1 #2\quad}}#3}
\newcommand\@dotsep{4.5}
\def\@tocline#1#2#3#4#5#6#7{\relax
	\ifnum #1>\c@tocdepth 
	\else
	\par \addpenalty\@secpenalty\addvspace{#2}%
	\begingroup \hyphenpenalty\@M
	\@ifempty{#4}{%
		\@tempdima\csname r@tocindent\number#1\endcsname\relax
	}{%
		\@tempdima#4\relax
	}%
	\parindent\z@ \leftskip#3\relax \advance\leftskip\@tempdima\relax
	\rightskip\@pnumwidth plus1em \parfillskip-\@pnumwidth
	#5\leavevmode\hskip-\@tempdima{#6}\nobreak
	\leaders\hbox{$\m@th\mkern \@dotsep mu\hbox{.}\mkern \@dotsep mu$}\hfill
	\nobreak
	\hbox to\@pnumwidth{\@tocpagenum{\ifnum#1=1\bfseries\fi#7}}\par
	\nobreak
	\endgroup
	\fi}
\AtBeginDocument{%
	\expandafter\renewcommand\csname r@tocindent0\endcsname{0pt}
}
\def\l@subsection{\@tocline{2}{0pt}{2.5pc}{5pc}{}}
\makeatother

\makeatletter
\renewcommand\section{\@startsection {section}{1}{\z@}%
	{-3.5ex \@plus -1ex \@minus -.2ex}%
	{2.3ex \@plus.2ex}%
	{\normalfont\@secnumfont\fontsize{15}{18}\bfseries}}
\renewcommand\subsection{\@startsection{subsection}{2}{\z@}%
	{-3.25ex\@plus -1ex \@minus -.2ex}%
	{1.5ex \@plus .2ex}%
	{\normalfont\fontsize{13.5}{17}\selectfont}}
\renewcommand\subsubsection{\@startsection{subsubsection}{3}{\z@}%
	{-3.25ex\@plus -1ex \@minus -.2ex}%
	{1.5ex \@plus .2ex}%
	{\normalfont\normalsize\fontsize{12.5}{17}\selectfont}}

\def\@seccntformat#1{%
	\protect\textup{\protect\@secnumfont
		\ifnum\pdfstrcmp{section}{#1}=0 \bfseries\fi
		\csname the#1\endcsname
		\protect\@secnumpunct
	}%
}  
\makeatother

\let\origmaketitle\maketitle
\def\maketitle{
	\begingroup
	\def\uppercasenonmath##1{} 
	\let\MakeUppercase\relax 
	\origmaketitle
	\endgroup
}

\newcommand{\cF}{{\mathcal F}}

\newcommand{\cH}{{\mathcal H}}

\newcommand{\cN}{{\mathcal N}}

\newcommand{\cR}{{\mathcal R}}

\newcommand{\cT}{{\mathcal T}}

\newcommand{\cY}{{\mathcal Y}}

\newcommand{\bC}{\mathbb C}

\newcommand{\bL}{{\mathbb L}}

\newcommand{\bN}{{\mathbb N}}

\newcommand{\bR}{{\mathbb R}}

\newcommand{\sF}{{\mathscr F}}

\newcommand{\sH}{{\mathscr H}}

\newcommand{\sN}{{\mathscr N}}

\newcommand{\sW}{{\mathscr W}}

\newcommand{\rE}{{\mathrm E}}

\newcommand{\rP}{{\mathrm P}}

\newcommand{\rd}{{\mathrm d}}
\renewcommand{\kappa}{\varkappa}

\title{\Large{Density Deconvolution with   Non--Standard Error Distributions: Rates of Convergence and 
Adaptive Estimation$^\ast$ 
\footnote{$^{\ast}$The research was supported by the Israel Science Foundation (ISF) research grant.}
}}

\author{Alexander Goldenshluger
\ \  \& \ \ Taeho Kim
}
\address{Department of Statistics\\ University of Haifa\\
Haifa 3498838,
Israel}
\curraddr{}
\email{goldensh@stat.haifa.ac.il; ktaeho@campus.haifa.ac.il}
\thanks{}

\keywords{}

\dedicatory{}

\begin{document}

\maketitle



\begin{abstract}
It is a 
standard assumption in the density deconvolution problem that 
the characteristic function of the measurement error distribution is non-zero on the real line. 
While this condition is assumed in the majority of existing works on the topic, 
there are many problem instances of interest  where it is violated. 
In this paper we focus on non--standard settings where the characteristic function of the measurement errors has
zeros, and study how zeros multiplicity affects the estimation accuracy. 
For a prototypical problem of this type we 
demonstrate that the best achievable estimation accuracy is determined by the multiplicity of zeros, the rate of decay of the  error characteristic function, as well as  by the smoothness and the tail behavior 
of the estimated density. 
We 
derive lower bounds on the minimax risk and 
develop optimal in the minimax sense estimators. In addition, we consider the problem 
of adaptive estimation and propose a data--driven  estimator that automatically adapts to 
unknown smoothness and tail behavior of the density to be estimated. \\
\vspace{3pt}
\\

\noindent\textbf{\textit{2010 AMS subject classification:}} 62G07, 62G20
\vspace{2pt}

\noindent\textbf{\textit{Keywords and phrases:}} Density Deconvolution, Minimax Risk, Characteristic Function, Laplace Transform, Non-standard Measurement Error, Zero Multiplicity
\end{abstract}
\newpage

\begin{spacing}{0.01}
\tableofcontents
\end{spacing}

\onehalfspacing

\section{Introduction}

Density deconvolution is a problem of estimating a  probability density from observations  
with  additive measurement errors. 
Specifically, assume that we observe random sample $Y_1, \ldots, Y_n$ generated by the model
\begin{align*}
Y_i&=X_i+\epsilon_i,\;\;
\;\;i=1,2,\ldots,n,	
\end{align*}
 where $X_i$'s are i.i.d.\ random variables with unknown density $f$  with respect to the Lebesgue measure
 on $\mathbb{R}$, 
 $\epsilon_i$'s are i.i.d.\ measurement errors with distribution function $G$, and $X_i$'s are independent of
 $\epsilon_i$'s. The objective is to estimate $f$ on the basis of the sample 
$\cY_n:=\{Y_1,\ldots, Y_n\}$.
 Since  $Y_i$ is the sum of two independent random variables, $X_i$ and $\epsilon_i$, 
  density $f_Y$ of $Y_i$ is given by 
the  convolution 
\begin{align}\label{MOD_2}
f_Y(y)=(f\star \rd G)(y)=\int^\infty_{-\infty} f(y-x) \rd G(x).
\end{align}
\par 
An estimator 
of the value of $f(x_0)$ is a measurable function of $\cY_n$,  $\hat{f}(x_0)=\hat{f}(x_0; \cY_n)$,
and 
 the risk of $\hat{f}(x_0)$ is 
\[
 \cR_{n}[\hat{f}, f] := \Big[\rE_f|\hat{f}(x_0)-f(x_0)|^2\Big]^{1/2},
\]
where $\rE_f$ stands for the expectation with respect to the probability measure $\rP_f$ generated by  the
observation $\cY_n$ when the unknown density of $X_i$'s is~$f$. For a particular functional class $\sF$, 
accuracy of $\hat{f}(x_0)$ is measured  by 
the {\em maximal risk} 
\[
 \cR_{n}[\hat{f}; \sF]:=\sup_{f\in \sF} \cR_{n}[\hat{f}, f],
\]
and an estimator $\hat{f}_*(x_0)$ is called {\em rate--optimal} or {\em optimal in order} on $\sF$ if 
\[
 \cR_{n}[\hat{f}_*; \sF]\asymp \cR_{n}^* [\sF]:=\inf_{\hat{f}}\cR_{n}[\hat{f}; \sF], \;\;\;n\to\infty.
\]
Here $\cR_{n}^*[\sF]$ is the {\em minimax risk}, and the infimum in its definition  is taken over all possible
estimators of $f(x_0)$. The objective 
in the density deconvolution problem is to construct an optimal in order estimator, and to study the rate
at which the minimax risk $\cR_{n}^*[\sF]$ converges to zero as $n\to\infty$.
In what follows we refer to the latter as {\em the minimax rate of convergence}.
\par 
The outlined problem is a subject of vast literature under various assumptions on the functional class~$\sF$ and 
distribution of measurement errors~$G$; see, e.g., 
 \citet{CH:1988}, \citet{SC:1990}, 
\citet{Zhang:90}, \citet{Fa:1991}, \citet{Butucea-Tsybakov-1, Butucea-Tsybakov-2}, 
\citet{Meister:09}, \citet{Nickl} 
for representative publications, where
further references can be found.
Typically $\sF$ is a class of functions satisfying 
smoothness conditions (e.g.,  H\"older or Sobolev functional classes).  
As for assumptions on the measurement error distribution, they are usually put in terms of
the characteristic function of $G$ and read as follows.
 \begin{assume}[\textbf{E0}]\label{ASM_E0} Let 
 $\phi_g(i\omega):=\mathcal{F}[\rd G,\omega]:=\int^\infty_{-\infty} e^{-i\omega x}\rd G(x)$ be the characteristic 
 function (the Fourier transform) of the measurement error distribution $G$. Then,
 \begin{itemize}
\item[I.] $|\phi_g(i\omega)|\ne 0$ for all $\omega\in\mathbb{R}$.
\item[II.] $|\phi_g(i\omega)|$ decreases at  polynomial or exponential rate as $|\omega|\to\infty$:
\\
 {\em ordinary smooth errors:} $|\phi_g(i\omega)|\asymp |\omega|^{-\gamma}$, $|\omega|\to \infty$ for some $\gamma>0$, or
\\
 {\em super-smooth errors:} $|\phi_g(i\omega)|\asymp \exp\{-c|\omega|^{\gamma}\}$, $|\omega|\to\infty$ for some $c>0$ and $\gamma>0$.
\end{itemize}
\end{assume}
Assumption~(E0) is inarguably conventional and presumed in nearly all 
works dealing with density deconvolution problems.
Under Assumption~(E0) accuracy in estimating $f$ is detemined by the rate at which 
$\phi_g$ tends to zero and by smoothness of $f$ as characterized in terms of functional class $\sF$. 
Condition~(E0-I) ensures 
that the statistical model is identifiable (it is well known that if $\phi_g$ vanishes on a set of non--zero Lebesgue measure
then $f$ is not identifiable). It   underlies applicability 
of the standard Fourier--transform--based techniques for constructing estimators
of~$f$. Note
however that (E0-I) does not hold 
if $\phi_g$ has isolated zeros  which is the case  
in many interesting situations, e.g., 
for continuous 
distributions  with compactly supported densities or for
general discrete distributions. 
For example, if $G$ is a uniform distribution on $[-1,1]$ 
then $\phi_g(i\omega)=\sin \omega/\omega$  has zeros at $\omega=\pm \pi k$, $k\in \bN$, and 
(E0-I) is not fulfilled. 
\par 
The settings in which the error characteristic function $\phi_g$ 
may have isolated zeros have been studied to a considerably  
lesser extent;  the available
results in this area are  fragmentary and disparate.  
\citet{De:1989} pointed out that 
 density~$f$ can be estimated consistently in the $\bL_1$--norm when  
the characteristic function $\phi_g$ of 
the error distribution is non-zero almost everywhere.
Although it is a quite general result, the convergence is not uniform, and 
the evaluation procedure is not based on the minimax criterion. 
Several previous studies investigated the  problem with the uniform error distribution.
In particular, \citet{GJ:2003} and \citet{Fe:2008} demonstrate that zeros of the characteristic 
function $\phi_g$ 
do not have influence on the 
minimax rate of convergence: it remains the same as under condition~(E0-I) when  
the estimated density $f$
is supported on the positive real line \cite{GJ:2003}, or has bounded second moment~\cite{Fe:2008}.
Considering a more general class of so-called
Fourier--oscillating error distributions, \citet{DM:2011} derive a similar result
for densities $f$ having finite left endpoint. 
In contrast to the aforementioned results, \citet{HM:2007} demonstrate  that 
for the class of Fourier--oscillating error distributions 
 zeros of the error characteristic function lead 
to a slower minimax convergence rate  than the one  under condition~(E0-I).
\citet{HM:2007} suggest 
a ``ridge'' modification of the kernel density deconvolution estimator in which 
characteristic function of the error distribution is regularized to avoid singularities  
due to the zeros.
For another closely related work we also refer to \citet{Me:2007}.
\par 
Recently a principled  method for  solving density deconvolution problems  under general assumptions 
on the 
error characteristic function has been proposed in \citet{BG:2019}.
This method uses  the Laplace transform (the Fourier transform in complex domain) in conjunction with the linear functional
strategy for constructing rate--optimal kernel deconvolution estimators. The results show 
that zeros of the error characteristic function have no influence on 
the achievable estimation accuracy  when, in addition  to usual smoothness conditions, the estimated density 
$f$ has  sufficiently light tails.  On the other hand, if $f$ is heavy tailed then zeros of the error characteristic function 
{\em affect} 
the minimax rates of convergence that become slower. 
\citet{BG:2019} provide an explicit condition on the tail behavior of $f$ and zeros geometry of $\phi_g$
under which the minimax rates of convergence are not influenced by the zeros of~$\phi_g$. 
\par 
In this paper we focus   on the setting when $\phi_g$ has zeros, and 
$f$ is heavy tailed relative to  the multiplicity $m$ of zeros of $\phi_g$ on the imaginary axis.
The prototypical settings of this type arise when mesurement error distribution is  the binomial
distribution 
${\rm Bin}(m, 1/2)$  or  the $m$--fold convolution of 
uniform distributions on $[-\theta, \theta]$.
Utilizing the methodology proposed in \cite{BG:2019} we develop rate--optimal estimators of $f$
and investigate their properties. 
It is shown that, 
in contrast to the well known results under Assumption~(E0), in the considered regime 
the minimax rate of convergence is determined 
not only by the smoothness of $f$ and the rate at which $\phi_g$ tends to zero, but also 
by the tail behavior of $f$ and 
the zero multiplicity of $\phi_g$.  
The derived lower bounds on the minimax risk demonstrate that dependence of the estimation accuracy on 
these factors  is essential.
\par 
The construction of the proposed rate--optimal estimator of $f$ depends on tuning
parameters, and their specification requires prior information on smoothness  and tail behavior of $f$.  
In practice such information is rarely available. To overcome this difficulty we propose and study 
an adaptive estimator  of $f$ that is based on the methodology developed in   
\citet{GL:2011,GL:2014}.
An interesting feature of the proposed estimator is that it involves two tuning parameters, and 
the adaptation here is not only with respect to the unknown smoothness, 
but also with respect to the unknown tail behavior of $f$. We derive an oracle inequality for the developed adaptive estimator 
and show that it achieves the minimax rate of convergence up to a logarithmic factor which is unavoidable 
payment for adaptation in 
point-wise estimation. 
\par 
The rest of the paper is organized as follows. 
In Section~\ref{sec:est-construction}  we present the general idea for estimator construction and introduce our estimator.
 Section~\ref{sec:minimax}  deals with minimax estimation of $f(x_0)$ with respect to proper 
 functional classes. In Section~\ref{sec:adaptive} we introduce the corresponding adaptive procedure and 
 investigate its properties. 
 Lastly, Section~\ref{sec:conclusion}  is reserved for discussion and concluding remarks. 
All the proofs are deferred to~Appendix.

 \section{Estimator Construction}\label{sec:est-construction}
\subsection{Idea of Construction}\label{EST_IDEA}
We start 
with presenting the key idea for estimator construction in our density 
deconvolution problem. 
The construction  uses Laplace transform (Fourier transform in the complex domain) which 
allows us to handle the situation where the first condition of Assumption~(E0) is not satisfied.
Our goal is to deliver the main idea of construction; for further details we refer to~\citet{BG:2019}. 
\par 
The following definitions will be utilized 
throughout the study. 
For a generic function $w$ the bilateral Laplace transform of $w$ 
is defined  to be  
\begin{align}\label{EQN_2.1}
	\mathcal{L}[w;z]:=\phi_w(z)=\int^\infty_{-\infty} w(x) e^{-zx}\rd x.
\end{align}
The integral convergence region $\Sigma_w$ (if exists) is a vertical strip  in the complex plane,  
$\Sigma_w=\{z\in \bC: Re(z)\in (\sigma_w^-, \sigma_w^+)\}$ for some $\sigma_w^-, \sigma_w^+\in \bR$, and 
$\phi_w(z)$ is analytic in $\Sigma_w$.
The inverse Laplace transform is 
\[
 w(x)=\frac{1}{2\pi i}\int_{s-i\infty}^{s+i\infty} \phi_w(z) e^{zx} \rd z= \frac{1}{2\pi}
 \int_{-\infty}^\infty \phi_w(s+i\omega) e^{(s+i\omega)x}\rd \omega,\;\;\;\;s\in (\sigma_w^-, \sigma_w^+).
\]
For the error distribution function $G$ we write  $\phi_g(z):=\int_{-\infty}^\infty e^{-zx}\rd G(x)$, and note
that the integral convergence region necessarily includes   the imaginary axis $\{z\in \bC: Re(z)=0\}$
with $\phi_g(i\omega)$ being the characteristic function of~$G$.
In what follows we assume that $\Sigma_g$ is a vertical strip in the complex plane, 
$\Sigma_g:=\{z\in \bC: Re(z)\in (\sigma_g^-, \sigma_g^+)\}$ for some $\sigma_g^-<0<\sigma_g^+$.
\par 
Our estimator uses a kernel whose construction relies upon {\em the linear functional strategy} for 
solution of ill-posed  problems (see, e.g., \cite{GB:1979}). Let $K\in C^\infty(\bR)$ be a kernel 
on $[-1,1]$ satisfying standard conditions:  for fixed $k\in \mathbb{Z}_+$ 
\begin{align}\label{eq:Kernel}
\int^1_{-1}K(t)dt=1,\;\;\;\;\int^1_{-1}t^jK(t)dt=0,\;\;\forall j=1,\ldots, k.	
\end{align}
Note that $\phi_K(z)$ is an entire function, i.e.\ $\Sigma_K=\mathbb{C}$. We would like to find a function
$L:\bR\to\bR$ such that for any given $x_0\in \bR$
\begin{equation}\label{eq:lin-fun}
 \int_{-\infty}^\infty L(y-x_0) f_Y(y)\rd y = \frac{1}{h}\int_{-\infty}^\infty K\Big(\frac{x-x_0}{h}\Big)f(x)\rd x,
\end{equation}
where we 
recall that $f_Y$ and $f$ are related to each other by the convolution integral (\ref{MOD_2}).
If function $L$ satisfying (\ref{eq:lin-fun}) is found then a reasonable estimator of $f(x_0)$ 
is given by the empirical 
estimator of the integral on the left hand side of (\ref{eq:lin-fun}) based on the sample $\cY_n$.
In our deconvolution problem this strategy is realized as follows.  
\par 
In addition to the analyticity of $\phi_g$ in $\Sigma_g$ we suppose
 that 
 $\phi_g(z)$ 
 does not vanish on the set 
 $\{z: Re(z)\in \left(\kappa_g^-,\kappa_g^+\right) \setminus \{0\}\}$ 
 for some 
 $\kappa_g^-$, $\kappa_g^+$ such that $\sigma_g^-\leq \kappa_g^-<0<\kappa_g^+\leq \sigma_g^+$. 
 Note that $\phi_g$ may have zeros on the imaginary axis $\{z: Re(z)=0\}$, so that the conventional Fourier transform 
 technique would not work in this situation.
 Let $S_g:=\left\{z:Re(z)\in \left(-\kappa_g^+,-\kappa_g^-\right)\setminus \{0\}\right\}$; in fact, 
 $S_g$ is the union of two open vertical strips in the complex plane having the imaginary axis as the 
 boundary. 
 Note that 
 $\phi_g(-z)\ne 0$ on $S_g$,
and for  $h>0$ define
\begin{align*}
\phi_{L}(z):=\frac{\phi_K(zh)}{\phi_g(-z)},	\;\;\;z\in S_g.
\end{align*}
Obviously,  $\phi_L$ is analytic on $S_g$, and we define  
kernel $L_{h}^s$ as the inverse Laplace transform of~$\phi_L$:
\begin{align}\label{EQN_2.3}
L_{h}^{s}(x):=\frac{1}{2\pi}\int^\infty_{-\infty}\frac{\phi_K((s+i\omega)h)}{\phi_g(-s-i\omega)}e^{(s+i\omega)x}d\omega,\;\;\;s\in (-\kappa_g^+, -\kappa_g^-)\setminus \{0\}.
\end{align}
Depending on the sign of parameter $s$ formula (\ref{EQN_2.3}) defines two different kernels which in the sequel are  denoted 
$L_h^+(\cdot)$ for $s>0$ and  $L_h^-(\cdot)$ for $s<0$. If the integral on the right hand side of 
(\ref{EQN_2.3}) is absolutely convergent and 
\begin{equation*}
\int^\infty_{-\infty}|L_{h}^s(y-x_0)|f_Y(y)dy < \infty,  
\end{equation*}
then by Lemma~1 in \cite{BG:2019} kernels $L_{h}^s$ and $K$ are related to each other via 
(\ref{eq:lin-fun}). Then we define  the resulting density deconvolution estimator by    
\begin{align*}
\hat{f}^s_{h}(x_0)=\frac{1}{n}\sum^n_{i=1}	L^s_{h}(Y_i-x_0),\;\;\;s\in (-\kappa_g^+, -\kappa_g^-)\setminus 
\{0\}.
\end{align*}
\par 
While a general form of the kernel $L^s_{h}$ is given  in (\ref{EQN_2.3}), it would be beneficial to 
specialize it for particular error distributions.  We handle this in the next 
subsection in relation to  error characteristic functions $\phi_g$   having zeros on the imaginary axis. 
\subsection{Measurement Error Distributions}
The following assumption on characteristic function of measurement errors has been introduced~in~\cite{BG:2019}.
\begin{assume}[\textbf{E1}]
	$\phi_g$ is analytic in $\Sigma_g:=\{z:Re(z)\in (\sigma_g^-,\sigma^+_g)\}$ with 
	$\sigma_g^-<0<\sigma_g^+$ and admits the following representation
	\begin{align}\label{EQN_2.5}
	\phi_g(z)=\frac{1}{\psi(z)}\prod^q_{k=1} \Big(1-e^{a_kz-ib_k}\Big)^{m_k},	
	\end{align}
where $\{a_k\}^q_{k=1}$ and $\{b_k\}_{k=1}^q$ are real numbers, $a_k> 0$, $b_k\in [0, 2\pi)$ for all $k$,
$\{m_k\}^q_{k=1}$ are non-negative integer numbers, and pairs $\left\{(a_k,b_k)\right\}^q_{k=1}$ are distinct 
for all $k$. The function $\psi(z)$ has the following representation:
\begin{align*}
	\psi(z)=\psi_0(z)\prod_{k: b_k=0}(-a_kz)^{m_k}\prod_{k: b_k\ne 0}(1- e^{-ib_k})^{m_k},
\end{align*}
where $\psi_0(z)$ is analytic and has no zeros in a vertical strip 
$\Sigma_\psi$, 
$\{z:Re(z)=0\}\subset \Sigma_\psi\subseteq\Sigma_g$.
\end{assume}
Assumption~(E1) postulates that characteristic function $\phi_g(z)$ is analytic in a vertical strip and 
can be factorized in a product of two functions: the first function has zeros on the imaginary axis while the second 
one does not vanish is the strip. Under (\ref{EQN_2.5}), the zeros of $\phi_g(z)$ are 
$z_{k,j}=i(b_k+2\pi j)/a_k$, $j=0, \pm 1, \pm 2, \ldots$, $z_{k,j}\ne 0$, and the multiplicity of  $z_{k,j}$ is equal to~$m_k$ for any~$j$. 
\par 
Assumption~(E1) is rather general.
It holds for  a wide class of discrete and continuous distributions
for specific examples we refer to  \cite[Section~3.2]{BG:2019}. 
Since the main focus of this study is to investigate the effect of zeros multiplicity of $\phi_g(z)$  
on the estimation accuracy,  we will concentrate on 
the following  prototypical examples: 
\begin{itemize}
 \item[(a)] [$m$--convolution of $U(-\theta, \theta)$ distribution]. Let $G$ be the distribution function of   
 $m$--fold convolution of 
 the uniform distribution on $[-\theta, \theta]$, $\theta>0$. 
 In this case 
 \begin{equation}\label{eq:phi-uniform}
 \phi_g(z)= \bigg[\frac{\sinh (\theta z)}{\theta z}\bigg]^{m}=   
 e^{-m\theta z}(-2\theta z)^{-m}(1-e^{2\theta z})^m,
\end{equation}
so that Assumption~(E1) holds with $q=1$, $a_1=2\theta$, $b_1=0$, $m_1=m$ and 
$\psi(z)= (-2\theta z)^m  e^{m\theta z}$.
\item[(b)]  [Binomial distribution]. Let $G$ be the distribution function of the binomial random variable with parameters $m$ and $p=1/2$;
then 
\begin{equation}\label{eq:phi-binomial}
 \phi_g(z)= 2^{-m} (1+e^z)^m,
\end{equation}
so that Assumption~(E1) holds with $q=1$, $a_1=1$, $b_1=\pi$ and $\psi(z)=2^m$.
\end{itemize}
\subsection{Estimator and Zero Multiplicity}
Under Assumption~(E1) the kernel in (\ref{EQN_2.3}) 
takes the following particular form:
\begin{align}\label{EQN_2.7}
	L^s_{h}(t)=\frac{1}{2\pi}\int^\infty_{-\infty}\frac{\phi_K((s+i\omega)h)\psi(-s-i\omega)}{\prod^q_{k=1}(1-e^{a_k(s+i\omega) -ib_k})^{m_k}} e^{(s+i\omega)t} \rd\omega,\;\;\;s+i\omega\in S_g.
\end{align}
While the denominator does not vanish for $s\in (-\kappa^+_g, \kappa^-_g)\setminus \{0\}$, 
the kernel representation is either $L^+_h$ or $L^-_h$, depending on the sign of~$s$.  
For examples (a)~and~(b) discussed above we can substitute expressions for $\phi_g(z)$ 
given by (\ref{eq:phi-uniform}) and (\ref{eq:phi-binomial})
in (\ref{EQN_2.3}).
Then  expanding formally the integrand in series (for details see \cite[Section~4.1]{BG:2019})
we come to the following infinite series representation
for the kernels:
\begin{itemize}
 \item[(a)] $m$--convolution of $U(-\theta, \theta)$ distribution:
 \begin{align*}
L^\pm_{h}(t)=\frac{(\pm2\theta)^m}{h^{m+1}}\sum_{j=0}^\infty C_{j,m} K^{(m)}\left(\frac{t\mp\theta (2j+m)}{h}\right); 
\end{align*}
 \item[(b)] binomial distribution:
 \begin{align*}
L^\pm_{h}(t)=\frac{(\pm 2)^m}{h}\sum_{j=0}^\infty C_{j,m}K\left(\frac{t\mp j}{h}\right),	
\end{align*}
\end{itemize}
where  
\[
C_{j,m}:={j+m-1\choose m-1}
\] 
is the number of weak compositions of $j$ into $m$ parts (see, e.g., 
\cite{ST:2011}). Note that the derived kernels $L_h^\pm$ are not integrable,  and, in general,
condition (\ref{eq:lin-fun}) is fulfilled
only if $f$ has sufficiently light tails. That is why in the estimator  construction 
we truncate the infinite series by a cut--off parameter $N$ coming to the kernels
\begin{align}
 & L^\pm_{h, N}(t):=
 \frac{(\pm2\theta)^m}{h^{m+1}}\sum_{j=0}^N C_{j,m} K^{(m)}\left(\frac{t\mp\theta (2j+m)}{h}\right),
\label{eq:L-N-unif}
 \\
& L^\pm_{h, N}(t):=\frac{(\pm 2)^m}{h}\sum_{j=0}^N C_{j,m}K\left(\frac{t\mp j}{h}\right)
 \label{eq:L-N-binomial}
 \end{align}
for examples~(a) and (b) respectively.
\par 
The multiplicity of zeros clearly manifests itself in construction of kernel $L_{h, N}^\pm$:
in setting~(a) multiplicity $m$ determines ill--posedness of the deconvolution problem, and in the 
both settings 
coefficients $C_{j,m}$ in 
(\ref{eq:L-N-unif}) and (\ref{eq:L-N-binomial}) grow with $m$ affecting the variance of the corresponding estimators
in the case of heavy tailed densities $f$.
Intuitively, the larger multiplicity  $m$, the flatter 
the characteristic function $\phi_g(z)$ in the vicinity of zeros, and the harder the deconvolution problem. 
\par 
Based on the derived kernels we define the estimators of $f(x_0)$ in examples~(a) and~(b) by 
\begin{align}
{\rm (a)}\;\;\; &\;\;\; \hat{f}^\pm_{h,N}(x_0)
=\frac{1}{n}\sum_{i=1}^n\frac{(\pm2\theta)^m}{h^{m+1}}\sum_{j=0}^NC_{j,m}K^{(m)}\left(\frac{Y_i-x_0\mp\theta(2j+m)}{h}\right),	
\label{EQN_2.12}
\\
{\rm (b)}\;\;\; &\;\;\; \hat{f}^\pm_{h,N}(x_0)
=\frac{1}{n}\sum_{i=1}^n\frac{(\pm2)^m}{h}\sum_{j=0}^NC_{j,m}K\left(\frac{Y_i-x_0\mp j}{h}\right),	
\label{EQN_2.12(b)}
\end{align}
where
$h$ and  $N$ are two tuning parameters that should be specified.
 \section{Minimax Results}\label{sec:minimax}
 In this section we derive upper bounds on the risk of the estimators constructed in the previous section,
and show that they are rate optimal over functional classes 
characterized by the smoothness and tail conditions. The analysis of the risk for the both estimators in cases~(a) and~(b) coincides in almost every detail.
Therefore in the sequel we concentrate on the example~(a); the corresponding results for binomial
error distribution are discussed in Section~\ref{sec:conclusion}.

 \subsection{Functional Classes}
 The following assumption introduces the 
 functional class over which accuracy of $\hat{f}_{h, N}^\pm (x_0)$ will be assessed. 
 \begin{assume}[\textbf{F}]
 	Let $A$ and $B$ be a positive real numbers.  
 	\begin{itemize}
 	\item[(I)] For $\alpha>0$, a probability density $f$ belongs to the functional class $\mathscr{H}_\alpha(A)$ if $f$ is $\lfloor\alpha\rfloor:=\max\{n\in\mathbb{N}\cup \{0\}:n<\alpha\}$ times continuously differentiable, and 
 	\begin{align}\label{EQN_3.1}
 	\left|f^{(\lfloor\alpha\rfloor)}(t)-f^{(\lfloor\alpha\rfloor)}(t')\right|\le A|t-t'|^{\alpha-\lfloor\alpha\rfloor},\;\;\;\forall t,\;t'\in\mathbb{R}
 	\end{align}
 	\item[(II)] Let $q$ be a positive real number. We say that a probability density 
 	$f$ belongs to the functional class $\mathscr{N}_q(B)$ if 
	\begin{align}\label{EQN_3.2}
 	f(t)\leq B|t|^{-q},\;\;\;\forall t\in \bR.	
 	\end{align}
	 \end{itemize}
	 Combining the two conditions in Assumption~(F), we define the following functional class: 
\begin{align*}
\mathscr{W}_{\alpha,q}(A,B):= \mathscr{H}_\alpha(A)\cap\mathscr{N}_q(B).
\end{align*}
 \end{assume}
\begin{remark}
While first assumption defines  the usual H\"older class $\sH_\alpha(A)$, the second condition imposes
a uniform upper bound on the decay of the tails 
of the measurement error density. 
Note that this tail condition is comparable to the moment condition in \cite[Definition 3]{BG:2019}. 
\end{remark}

 \subsection{Rates of Convergence}
Now we are in a position to establish upper bounds on the maximal risk of the estimator
$\hat{f}_{h, N}^\pm (x_0)$ defined in (\ref{EQN_2.12}). Let 
\begin{equation}\label{eq:est-1}
 \hat{f}_{h, N}(x_0):=\left\{\begin{array}{ll}
                              \hat{f}^+_{h,N}(x_0), & x_0\geq 0,\\
                              \hat{f}^-_{h,N}(x_0), & x_0<0,
                             \end{array}\right.
\end{equation}
\begin{equation}\label{eq:r-nu}
 r:= \left\{\begin{array}{ll}
(\alpha/q)(2m-1-q), & q<2m-1,\\
0, & q\geq 2m-1,
  \end{array} \right.,\;\;\; 
\nu:= 
\frac{\alpha}{2\alpha+2m+1+r},
  \end{equation}
and  define
\begin{equation}\label{eq:varphi}
 \varphi(n):= \left\{
 \begin{array}{ll} 
 \big(B^{1/\alpha}A^{{\frac{2m+1}{\alpha}}}\big)^\nu n^{-\nu}, & q > 2m-1, \\*[3mm]
\big(B^{1/\alpha}A^{\frac{2m+1}{\alpha}}\big)^\nu \big(\frac{\log n}{n}\big)^\nu, & q=2m-1, \\*[3mm]
\big(B^{\frac{2m-1}{\alpha q}}A^{\frac{2m+1}{\alpha}}\big)^\nu  n^{-\nu}, & q<2m-1.
                    \end{array}
\right.
\end{equation}
 \begin{restatable}{theorem}{first}
\label{THM_1}
Let $f\in \mathscr{W}_{\alpha,q}(A,B)$ with $q>0$, and let $\phi_g(z)=[\sin(\theta z)/(\theta z)]^m$, 
$m\in \mathbb{N}$. Let 
$\hat{f}_{h,N}(x_0)$ be the estimator defined 
in (\ref{eq:est-1}) and (\ref{EQN_2.12})  and associated with kernel $K$ satisfying condition (\ref{eq:Kernel}) with parameter
$k\geq \alpha+1$. Then with $h=h_*$ and $N=N_*$ defined in (\ref{eq:q>})--(\ref{eq:q<}) in the proof of the theorem 
one has
\begin{equation*}
 \limsup_{n\to\infty} \Big\{ [\varphi(n)]^{-1}\cR_n [\hat{f}_{h_*, N_*}; \sW_{\alpha, q}(A, B)]\Big\} \leq C_1,
\end{equation*}
where $C_1$ is a constant independent of $A$ and $B$. 
\end{restatable}
\begin{remark}\mbox{}
\begin{itemize}\item[(a)]
 The result of Theorem~\ref{THM_1} shows how the tail behavior of $f$ and  
 zeros multiplicity $m$ affect the estimation accuracy. If the tail of $f$ is sufficiently light, i.e., $q>2m-1$, then 
 the risk of $\hat{f}_{h_*, N_*}(x_0)$ converges to zero at the rate $n^{-\alpha/(2\alpha+2m+1)}$ which 
 is obtained in the ordinary smooth case with $\gamma=m$ and non--vanishing characteristic function 
 $\phi_g$  [see Assumption~(E0)]. On the other hand, for heavy tailed densities $f$ with  
 $q<2m-1$ the maximal risk of $\hat{f}_{h_*, N_*}(x_0)$ converges at a slower rate, and parameter $r$ 
in (\ref{eq:r-nu}) characterizes  deterioration in the convergence rate. 
\item[(b)] The existence of different regimes depending on the tail behavior of $f$
and zeros multiplicity $m$ has been noticed in \cite{BG:2019}; however, the case of heavy tailed densities
has not been studied there. 
\end{itemize}
\end{remark}
\par 
Next theorem provides a lower bound on the  minimax risk of estimation over functional class~$\sW_{\alpha, q}(A,B)$. 
\begin{restatable}{theorem}{second}
\label{THM_2}
Let $f\in \mathscr{W}_{\alpha,q}(A,B)$ for $q>1$ and $\phi_g(z)=[\sin(\theta z)/(\theta z)]^m$, $m\in \mathbb{N}$. Then 
\begin{align*}
\liminf_{n\to\infty} \Big\{ 
\big(A^{-(2m+1)/\alpha}\, n\big)^{\nu}\,
\mathcal{R}^*_n\left[\mathscr{W}_{\alpha,q}(A,B)\right] \Big\}\ge
C_2,
\end{align*}
where $\nu$ is defined in  (\ref{eq:r-nu}), and $C_2$ is a positive constant independent of $A$.
 \end{restatable}
  \begin{remark}\mbox{}
  \begin{itemize}
\item[(a)] Theorems~\ref{THM_1} and~\ref{THM_2} show that there are two regimes in behavior
of the minimax risk. These regimes are characterized  by the tail behavior of the estimated density $f$ and  
the  multiplicity of zeros of the error characteristic function $\phi_g$. In the {\em light tail regime}, 
$q>2m-1$, the zeros of $\phi_g$ have no influence on the minimax rate of convergence: it is fully determined by the tail 
behavior of $\phi_g$. On the other hand, if $q<2m-1$  (the  {\em heavy tail regime}) then zeros of $\phi_g$ 
have significant influence on 
the minimax rate, it becomes much slower than in the case of non--vanishing $\phi_g$. 
\item[(b)] Theorems~\ref{THM_1} and~\ref{THM_2} demonstrate that   
the proposed estimator $\hat{f}_{h_*, N_*}(x_0)$ is rate optimal in  both {\em light tail} and 
{\em heavy tail regimes}. We note that on the boundary $q=2m-1$ between two regimes 
there is a logarithmic gap between the upper and lower bounds of Theorems~\ref{THM_1} and~\ref{THM_2}.
\end{itemize}
 \end{remark}
 \par 
 Thus far, the risk evaluations are under the functional class 
 $\mathscr{W}_{\alpha,q}(A,B)$ defined in Assumption~(F). 
Although these conditions are pretty reasonable in the context of the density deconvolution, 
they involve an extra assumption on the  
tail behavior of~$f$, and it is natural to ask what happens when the tail condition does not hold. 
The next result provides an answer to this question.
%
\begin{corollary}
 \label{THM_3}
Let  $\phi_g(z)=[\sin(\theta z)/(\theta z)]^m$, $m\in \mathbb{N}$;
then the following results hold
\begin{align}
&\liminf_{n\to \infty}\Big\{
 \psi_n^{-1} \mathcal{R}^*_n
 [\mathscr{H}_{\alpha}(A)]\Big\}\ge C_3, 
\label{eq:lower-holder}
 \\
&\limsup_{n\to \infty}\Big\{
 \psi_n^{-1} \mathcal{R}^*_n
 [\mathscr{H}_{\alpha}(A)\cap \sN_1(B)]\Big\}\le C_4, 
 \label{eq:upper-holder}
 \end{align}
where $\psi_n:=(A^{(2m+1)/\alpha}/ n)^{\frac{\alpha}{2m\alpha+2m+1}}$,
and $C_3$ and $C_4$ do not depend on $A$.
 \end{corollary}
 \begin{remark}
 In view of  (\ref{eq:lower-holder}), the rate of convergence $\psi_n$  on the functional class $\sH_\alpha(A)$ 
 is significantly slower   than the one achieved on $\sH_\alpha(A)$ in the setting with non--vanishing 
 characteristic function  $\phi_g$. Note that 
 the upper bound in (\ref{eq:upper-holder}) is achieved on a slightly smaller functional class. 
 The assumption $f\in \sN_1(B)$
 is very mild and is fulfilled for virtually any probability density. However it  does not hold uniformly for all densities.  
 We were not able to derive the upper bound (\ref{eq:upper-holder}) without this additional condition.
\end{remark}

\section{Adaptive Procedure}\label{sec:adaptive}
The minimax results in the previous section can only be achieved when 
the information on the functional class is known to us in advance. This is evident by observing that 
the optimal choice of tuning parameters $h_*$ and $N_*$ requires knowledge of the functional class. 
However, in most of applications, it is extremely rare to have the advance information about the functional class where the target function $f$ resides in.  
Therefore, it is natural to ask whether one can construct an estimator with 
the equivalent or comparable accuracy guarantees 
without knowing the functional class parameters.  
\par 
In this section we develop an adaptive estimator of $f(x_0)$ whose construction 
is based on the idea of data--driven selection from a
family of estimators $\{\hat{f}_{h, N}(x_0): (h, N)\in \cH\times \cN\}$, where $\hat{f}_{h, N}(x_0)$ is
defined in the previous section, and $\cH$ and $\cN$ are some fixed sets of bandwidths and cut--off parameters.  
Since the estimators $\hat{f}_{h, N}(x_0)$ depend on two tuning parameters, 
we adopt the general method of adaptive estimation proposed 
in \cite{GL:2011}. 

%
\subsection{Selection Rule} 
Let $\cH$ and $\cN$ be the discrete sets defined as follows: 
for  real numbers $0<h_{\min}<h_{\max}=\theta$ and 
integer number $N_{\max}$ to be specified later
\begin{equation*}
 \cH:=\big\{h \in [h_{\min}, h_{\max}]:  h=2^{-j}h_{\max},\; j=0, \ldots, M_h \big\},\;\;\;
 \cN:=\big\{j: j=1, \ldots, N_{\max}=:M_N\big\},
\end{equation*}
where $M_h:= \lfloor \log_2 (h_{\max}/h_{\min})\rfloor$ and $M_N:=N_{\max}$ denote 
the cardinality of   $\cH$ and $\cN$ respectively. 
\par 
Let $\mathcal{T}:=\mathcal{H}\times\mathcal{N}$,  
define $\tau:=(h, N)$, and consider the family of estimators 
$\cF(\cT)=\{\hat{f}_\tau^\pm (x_0),\;\;\tau\in \cT\}$,
where $\hat{f}^\pm_\tau(x_0)=\hat{f}_{h, N}^\pm(x_0)$ is defined in 
(\ref{EQN_2.12}) and (\ref{eq:est-1}). The adaptive estimator is based on   data--driven selection from the 
family~$\cF(\cT)$. For the sake of definiteness in the sequel we assume that $x_0\geq 0$ and 
consider estimators $\hat{f}^+_\tau(x_0)$ only; the case $x_0<0$ and $\hat{f}^-_\tau(x_0)$ is handled 
in exactly the same way. 
\par 
The selection rule uses auxiliary estimators that are constructed as follows.
For  $\tau, \tau^\prime \in \cT$ let  $\tau\vw\tau':=	(h\vee h',N\wedge N')$  denote 
the operation of coordinate-wise maximum and minimum.
With any pair $\tau, \tau^\prime\in \cT$ we associate the estimator [cf.~(\ref{EQN_2.12})]  
\begin{align*}
\hat{f}^{+}_{\tau\vw\tau'}(x_0):=\frac{1}{n}\sum_{i=1}^n\frac{(2\theta)^m}{(h\vee h')^{m+1}}
\sum_{j=0}^{N\wedge N'}C_{j,m}K^{(m)}\left(\frac{Y_i-x_0 - \theta(2j+m)}{h\vee h'}\right).	
\end{align*}
Observe that $\hat{f}^+_{\tau \vw \tau^\prime}(x_0)=\hat{f}^+_{\tau^\prime\vw\tau}(x_0)$ for all $\tau, \tau^\prime\in \cT$.
\par 
	Selection rules based on convolution--type auxiliary kernel estimators are developed   in 
\cite{GL:2011,GL:2014}, while   
\citet{Lep:2015} uses auxiliary estimators that are based on the operation of 
point--wise maximum of multi--bandwidths. Our construction is close in spirit to the latter one; it is
dictated
 by the structure of estimators $\hat{f}^{\pm}_{h, N}(x_0)$  in the deconvolution problem.
\par 
An important ingredient in the construction of the proposed selection rule is a uniform upper bound 
 on the stochastic error of estimator $\hat{f}^+_\tau(x_0)$, $\tau\in \cT$. 
For $\tau\in \cT$  the stochastic error of $\hat{f}^+_\tau(x_0)$ is   
\begin{equation}\label{eq:xi}
 \xi_\tau(x_0):=\frac{1}{n} \sum_{i=1}^n  L^+_\tau(Y_i-x_0)-  \rE_f \big[L^+_\tau(Y_1-x_0)\big],
\end{equation}
where 
\[
 L^+_{\tau}(y):= \frac{(2\theta)^m}{h^{m+1}} 
 \sum_{j=0}^N C_{j,m} K^{(m)}\bigg(\frac{y-\theta(2j+m)}{h}\bigg);
\]
see (\ref{eq:L-N-unif}).
Define 
\begin{align}\label{eq:sigma-tau}
 \sigma_{\tau}^2 &:=\frac{(2\theta)^{2m}}{h^{2m+2}}\sum_{j=0}^NC_{j,m}^2\int_{-\infty}^\infty
 \left|K^{(m)}\left(\frac{y-x_0-\theta(2j+m)}{h}\right)\right|^2f_Y(y) \rd y.
\end{align}
The proof of Theorem~\ref{THM_1} shows that ${\rm var}_f\{\xi_\tau(x_0)\} \leq \sigma_\tau^2/n$.
Let  
\begin{equation}\label{eq:eta-u}
u_\tau:= 2^{m+1}\theta^m C_{N,m}\|K^{(m)}\|_\infty h^{-m-1},\;\;\;\;\;\;
\end{equation}
and  for real number $\kappa>0$ that will be specified later we put 
\begin{equation}\label{eq:Lambda-tau}
 \Lambda_\tau(\kappa):= \sigma_\tau \sqrt{\frac{2\kappa}{n}} + \frac{2u_\tau\kappa}{3n}.
\end{equation}
In Lemma~\ref{lem:xi} 
in Appendix we demonstrate that $\Lambda_\tau(\kappa)$ is a uniform upper bound on $|\xi_\tau(x_0)|$ in the sense
that all moments of the random variable $\sup_{\tau\in \cT} [|\xi_\tau(x_0)|- \Lambda_\tau(\kappa)]_+$ are suitably small
as $\kappa$ increases. Note however that $\Lambda_\tau(\kappa)$ cannot be used in  the selection rule because it depends on the unknown density. In order to overcome this problem 
we consider a data--driven uniform upper bound on 
$\xi_\tau(x_0)$ that is constructed as follows. 
\par 
For $\tau \in \cT$ let 
\[
 \hat{\sigma}_\tau^2:=
 \frac{1}{n} \sum_{i=1}^n \frac{(2\theta)^{2m}}{h^{2m+2}}\sum_{j=0}^N C_{j,m}^2 \,
  \bigg|K^{(m)}\bigg(\frac{Y_i-x_0-\theta(2j+m)}{h}\bigg)\bigg|^2. 
\]
Note that $\hat{\sigma}^2_\tau$ is the  empirical estimator of $\sigma_\tau^2$.
Let 
\begin{equation}\label{eq:Lambda-hat}
 \hat{\Lambda}_\tau (\kappa):= 7 \bigg(\hat{\sigma}_\tau \sqrt{\frac{2\kappa}{n}} +  \frac{2u_\tau \kappa}{3n}\bigg).
\end{equation}
With  the  introduced notation the selection rule is the following. For any $\tau\in \cT$ define
\begin{align}
\hat{R}_\tau(x_0):=\sup_{\tau^\prime\in \cT}
\left[\big|\hat{f}^+_{\tau\vw\tau'}(x_0)-\hat{f}^+_{\tau'}(x_0)\big|-\hat{\Lambda}_{\tau\vw\tau^\prime}(\kappa) 
 - \hat{\Lambda}_{\tau^\prime}(\kappa)\right]_+ 
\nonumber
\\
 \;+\; \hat{\Lambda}_\tau(\kappa) + \sup_{\tau^\prime\in \cT}\hat{\Lambda}_{\tau\vw \tau^\prime}(\kappa).
 \label{EQN_4.4}
\end{align}	
Then, the adaptive estimator $\hat{f}_*(x_0)$ is defined by
\begin{align}\label{EQN_4.5}
\hat{f}_*(x_0):=\hat{f}^+_{\hat{\tau}}(x_0),\;\;\;
\hat{\tau}=\big(\hat{h},\hat{N}\big):=\argmin_{\tau \in\mathcal{T}}\hat{R}_{\tau}(x_0).
\end{align}
\begin{remark}
The defined selection rule is fully data--driven; it only 
requires specification of
parameter~$\kappa$ in (\ref{eq:Lambda-hat}). This parameter provides a uniform control of 
the stochastic errors for the family of estimators $\cF(\cT)$, and has no relation 
to the properties of the density to be estimated.  In addition, the 
parameters $h_{\min}$ and $N_{\max}$ should be chosen; they  
determine the sets of admissible bandwidths $\cH$ and 
cut--off parameters $\cN$.
\end{remark}

\subsection{Oracle Inequality and Rates of Convergence}
For $h, h^\prime \in \cH$ and $N, N^\prime \in \cN$ define 
\begin{eqnarray}
\bar{B}_h(f) &:=&   \sup_{h^\prime \leq h}\sup_{x\in \bR}\bigg|\frac{1}{h}
\int_{-\infty}^\infty K\Big(\frac{t-x}{h}\Big) [f(t)-f(x)] \rd t\bigg|,
\label{eq:barB-h}
\\
 \bar{B}_N(x_0; f)&:=&  \max_{1\leq j\leq m}\sup_{|t|\leq \theta} 
 \sup_{N^\prime \geq N} \big[f\big(t+x_0+2\theta (N^\prime+1)j\big)\big],
\label{eq:barB-N}
 \end{eqnarray}
 and let 
 \begin{equation}\label{eq:bias-UB}
 \bar{B}_\tau(x_0; f):= 2^{m+1} \Big[\bar{B}_h(f) +
(1+\|K\|_1) \bar{B}_N(x_0; f)\Big].
\end{equation}

\begin{theorem}\label{th:oracle-ineq}
 Let $\hat{f}_{*}(x_0)$ be the estimator defined in  (\ref{EQN_4.4})-(\ref{EQN_4.5}) and associated 
 with parameter $\kappa>0$; then 
 \begin{align}\label{eq:oracle}
  |\hat{f}_*(x_0)-f(x_0)| \leq C_1 \inf_{\tau\in \cT} 
  \Big\{\bar{B}_\tau (x_0;f) +  \Lambda_\tau(\kappa) \Big\} + C_2\Big(\delta(x_0) + \frac{\kappa}{n}\Big),
 \end{align}
where $C_1$ is an absolute constant, $C_2$ depends only on $m$ and $\theta$, and $\delta(x_0)$ is a 
non--negative random variable
that admits the following bound: for any $p\geq 1$
\begin{equation}\label{eq:delta}
 \rE_f \big[\delta(x_0)\big]^p \leq C_3 M_h M_N [\bar{\Lambda}(\kappa)]^p 
\kappa^{-p} e^{-\kappa}, \;\;\;
\end{equation}
where $\bar{\Lambda}(\kappa):= \sup_{\tau\in \cT} \{(1+u_\tau) \Lambda_\tau(\kappa)\}$, and 
constant $C_3$ depends on $p$ only.

\end{theorem}

\begin{remark}
 Explicit  expressions for constants $C_1$, $C_2$ and $C_3$ appear in the proofs of Theorem~\ref{th:oracle-ineq}
 and Lemma~\ref{lem:Lambda-Lambda-tilde}. Note that the oracle inequality holds for any probability density $f$, without  any functional class assumptions.
\end{remark}
\par 
The oracle inequality (\ref{eq:oracle}) allows us to derive the following result on the accuracy
of the adaptive estimator $\hat{f}_*(x_0)$ on the class $\sW_{\alpha, q}(A, B)$. 
\begin{corollary}\label{cor:2}
 Suppose that $f\in \sW_{\alpha, q}(A, B)$ with $q\geq 1$. 
 Let $\cF(\cT)$ be the family of estimators $\big\{\hat{f}^+_{h, N}(x_0), (h, N)\in \cH\times \cN\big\}$ with  
 \begin{equation}\label{eq:hmin-Nmax}
  h_{\min}:= \Big(\frac{\log n}{n}\Big)^{1/(2m+1)},\;\;h_{\max}=\theta,\;\;
  N_{\max}:=\Big(\frac{n}{\log n}\Big)^{1/(2m)}.
 \end{equation}
 Let $\hat{f}_*(x_0)$ be the estimator 
 defined  by selection rule  (\ref{EQN_4.4})-(\ref{EQN_4.5}) and associated with 
 parameter~$\kappa=\kappa_*:= 5 \log n$; then 
  \[
 \limsup_{n\to\infty} \Big\{ \Big[\varphi\Big(\frac{n}{\log n}\Big)\Big]^{-1}\cR_n [\hat{f}_{h_*, N_*}; \sW_{\alpha, q}(A, B)]\Big\} \leq C,
\]
where $\varphi(\cdot)$ is defined in (\ref{eq:varphi}), and $C$ does not depend on $A$ and $B$.
\end{corollary}
\begin{remark}
Note that the resulting rate is the same as the rate of convergence in Theorem~\ref{THM_1} 
except for the extra $\log n$ factor. It is a well-known fact by \citet{L:1991} that this factor 
cannot be avoided in the adaptive nonparametric estimation of a function at a single point.  
\end{remark}

\section{Concluding Remarks}\label{sec:conclusion}
We close  this paper with a few concluding remarks.
\par
In this paper 
 we concentrated on the setting when the error distribution is the $m$--fold convolution of the  
uniform distribution on $[-\theta, \theta]$. 
Here the error characteristic function has infinite number of isolated zeros
on the imaginary axis, each of them has the same multiplicity~$m$. 
Note that the results  of Theorems~\ref{THM_1},~\ref{THM_2}, and Corollary~\ref{cor:2}
also hold for the binomial error 
distribution ${\rm Bin}(m, 1/2)$ with the following minor changes in notation: 
in (\ref{eq:r-nu}) parameter $\nu$ should be 
redefined as $\nu= 1/(2\alpha+1+r)$, and in~(\ref{eq:varphi}) and in the statement of Theorem~\ref{THM_2}
 expression  $A^{(2m+1)/\alpha}$ should be replaced by $A^{1/\alpha}$.
The specific 
form of the error characteristic functions used in this paper 
facilitates derivation of lower bounds on the minimax risk. However, in general, 
the proposed technique is applicable 
to other error distributions whose charatceristic function has zeros on the imaginary axis.  
\par 
We developed rate optimal estimators with respect to the   point--wise  risk.   
It is worth noting there is a significant difference between settings with  point--wise and
$\bL_2$--risks when the error characteristic function has zeros on the imaginary axis. 
This fact has been already noticed in \cite{BG:2019}. Some results for density deconvolution 
with $\bL_2$--risk 
for non--standard error distributions appeared in  \cite{Me:2007} and \cite{HM:2007}. 
In general, deconvolution problems under  global losses with non--standard error distributions
deserve a thorough study.

\bibliographystyle{plainnat}
\bibliography{main.bib}

\begin{thebibliography}{22}
\providecommand{\natexlab}[1]{#1}
\providecommand{\url}[1]{\texttt{#1}}
\expandafter\ifx\csname urlstyle\endcsname\relax
  \providecommand{\doi}[1]{doi: #1}\else
  \providecommand{\doi}{doi: \begingroup \urlstyle{rm}\Url}\fi

\bibitem[Aubin(2000)]{Au:2011}
Jean-Pierre Aubin.
\newblock \emph{Applied functional analysis}.
\newblock John Wiley \& Sons, 2 edition, 2000.

\bibitem[Belomestny and Goldenshluger(2019)]{BG:2019}
Denis Belomestny and Alexander Goldenshluger.
\newblock Density deconvolution under general assumptions on the distribution
  of measurement errors.
\newblock \emph{arXiv preprint arXiv:1907.11024}, 2019.

\bibitem[Butucea and Tsybakov(2007{\natexlab{a}})]{Butucea-Tsybakov-1}
Cristina Butucea and Alexandre~B Tsybakov.
\newblock Sharp optimality in density deconvolution with dominating bias. {I}.
\newblock \emph{Teoriya Veroyatnoste\u{\i} i ee Primeneniya}, 52\penalty0
  (1):\penalty0 111--128, 2007{\natexlab{a}}.

\bibitem[Butucea and Tsybakov(2007{\natexlab{b}})]{Butucea-Tsybakov-2}
Cristina Butucea and Alexandre~B Tsybakov.
\newblock Sharp optimality in density deconvolution with dominating bias. {II}.
\newblock \emph{Teoriya Veroyatnoste\u{\i} i ee Primeneniya}, 52\penalty0
  (2):\penalty0 336--349, 2007{\natexlab{b}}.

\bibitem[Carroll and Hall(1988)]{CH:1988}
Raymond~J Carroll and Peter Hall.
\newblock Optimal rates of convergence for deconvolving a density.
\newblock \emph{Journal of the American Statistical Association}, 83\penalty0
  (404):\penalty0 1184--1186, 1988.

\bibitem[Delaigle and Meister(2011)]{DM:2011}
Aurore Delaigle and Alexander Meister.
\newblock Nonparametric function estimation under {F}ourier-oscillating noise.
\newblock \emph{Statistica Sinica}, 21\penalty0 (3):\penalty0 1065--1092, 2011.

\bibitem[Devroye(1989)]{De:1989}
Luc Devroye.
\newblock Consistent deconvolution in density estimation.
\newblock \emph{The Canadian Journal of Statistics/La Revue Canadienne de
  Statistique}, 17\penalty0 (2):\penalty0 235--239, 1989.

\bibitem[Fan(1991)]{Fa:1991}
Jianqing Fan.
\newblock On the optimal rates of convergence for nonparametric deconvolution
  problems.
\newblock \emph{The Annals of Statistics}, 19\penalty0 (3):\penalty0
  1257--1272, 1991.

\bibitem[Feuerverger et~al.(2008)Feuerverger, Kim, and Sun]{Fe:2008}
Andrey Feuerverger, Peter~T Kim, and Jiayang Sun.
\newblock On optimal uniform deconvolution.
\newblock \emph{Journal of Statistical Theory and Practice}, 2\penalty0
  (3):\penalty0 433--451, 2008.

\bibitem[Golberg(1979)]{GB:1979}
Michael~A Golberg.
\newblock A method of adjoints for solving some ill-posed equations of the
  first kind.
\newblock \emph{Applied Mathematics and Computation}, 5\penalty0 (2):\penalty0
  123--129, 1979.

\bibitem[Goldenshluger and Lepski(2011)]{GL:2011}
Alexander Goldenshluger and Oleg Lepski.
\newblock Bandwidth selection in kernel density estimation: oracle inequalities
  and adaptive minimax optimality.
\newblock \emph{The Annals of Statistics}, 39\penalty0 (3):\penalty0
  1608--1632, 2011.

\bibitem[Goldenshluger and Lepski(2014)]{GL:2014}
Alexander Goldenshluger and Oleg Lepski.
\newblock On adaptive minimax density estimation on {$\mathbb{R}^d$}.
\newblock \emph{Probability Theory and Related Fields}, 159\penalty0
  (3-4):\penalty0 479--543, 2014.

\bibitem[Groeneboom and Jongbloed(2003)]{GJ:2003}
Piet Groeneboom and Geurt Jongbloed.
\newblock Density estimation in the uniform deconvolution model.
\newblock \emph{Statistica Neerlandica}, 57\penalty0 (1):\penalty0 136--157,
  2003.

\bibitem[Hall and Meister(2007)]{HM:2007}
Peter Hall and Alexander Meister.
\newblock A ridge-parameter approach to deconvolution.
\newblock \emph{The Annals of Statistics}, 35\penalty0 (4):\penalty0
  1535--1558, 2007.

\bibitem[Lepski(1991)]{L:1991}
Oleg Lepski.
\newblock On a problem of adaptive estimation in {G}aussian white noise.
\newblock \emph{Theory of Probability \& Its Applications}, 35\penalty0
  (3):\penalty0 454--466, 1991.

\bibitem[Lepski(2015)]{Lep:2015}
Oleg Lepski.
\newblock Adaptive estimation over anisotropic functional classes via oracle
  approach.
\newblock \emph{The Annals of Statistics}, 43\penalty0 (3):\penalty0
  1178--1242, 2015.

\bibitem[Lounici and Nickl(2011)]{Nickl}
Karim Lounici and Richard Nickl.
\newblock Global uniform risk bounds for wavelet deconvolution estimators.
\newblock \emph{The Annals of Statistics}, 39\penalty0 (1):\penalty0 201--231,
  2011.

\bibitem[Meister(2007)]{Me:2007}
Alexander Meister.
\newblock Deconvolution from {F}ourier-oscillating error densities under decay
  and smoothness restrictions.
\newblock \emph{Inverse Problems}, 24\penalty0 (1):\penalty0 015003, 2007.

\bibitem[Meister(2009)]{Meister:09}
Alexander Meister.
\newblock \emph{Deconvolution problems in nonparametric statistics}, volume 193
  of \emph{Lecture Notes in Statistics}.
\newblock Springer-Verlag, Berlin, 2009.

\bibitem[Stanley(2011)]{ST:2011}
Richard~P Stanley.
\newblock \emph{Enumerative combinatorics}, volume~1 of \emph{Cambridge studies
  in advanced mathematics}.
\newblock Cambridge University Press, second edition, 2011.

\bibitem[Stefanski and Carroll(1990)]{SC:1990}
Leonard~A Stefanski and Raymond~J Carroll.
\newblock Deconvolving kernel density estimators.
\newblock \emph{Statistics}, 21\penalty0 (2):\penalty0 169--184, 1990.

\bibitem[Zhang(1990)]{Zhang:90}
Cun-Hui Zhang.
\newblock Fourier methods for estimating mixing densities and distributions.
\newblock \emph{The Annals of Statistics}, 18\penalty0 (2):\penalty0 806--831,
  1990.

\end{thebibliography}

\newpage

\begin{appendix}
\section{Proofs}

\subsection{Proof of Theorem~\ref{THM_1}}
\begin{proof}
In the subsequent proof $c_1, c_2, \ldots, $ stand for positive constants independent of $A$ and $B$.
Without loss of generality we assume that $x_0\geq 0$; the proof for the case $x_0<0$ is identical in every detail.
We follow the ideas  of the proof of Theorem~2 in \cite{BG:2019}. 
\par
(a). We begin with bounding the variance of $\hat{f}^+_{h, N}(x_0)$. 
It is shown in \cite{BG:2019} there that the variance of $\hat{f}_{h, N}^+(x_0)$ is bounded from above as follows
\begin{align}
 {\rm var}_f\big[\hat{f}^+_{h, N}(x_0)\big] \leq 
 \frac{(2\theta)^{2m}}{nh^{2m+2}} \sum_{j=0}^N C_{j,m}^2 
 \int_{-\infty}^\infty \Big|K^{(m)}\Big(\frac{y-x_0-\theta(2j+m)}{h}\Big)\Big|^2 f_{Y}(y)\rd y
 \nonumber
 \\
  \leq \frac{c_1\theta^{2m}}{nh^{2m+1}} 
 \sum_{j=0}^N \frac{C_{j,m}^2}{h}\int_{I_j(x_0)} f_Y(t) \rd t,
 \label{eq:var-1}
\end{align}
where $I_j(x_0):=[x_0+\theta(2j+m)-h, x_0+\theta(2j+m)+h]$.
Furthermore, by (A.16) in \cite{BG:2019},  
\begin{eqnarray*}
 \frac{1}{h} \int_{I_j(x_0)} f_Y(y)\rd y \leq \frac{c_2}{\theta}\int_{-h}^h f(t+x_0+2(j+m)\theta)\rd t
 +\frac{c_3}{\theta} \int_{-h}^h f(t+x_0+2j\theta)\rd t 
 \\
 + \frac{c_4}{\theta} \int_{-m\theta}^{m\theta} f(t+x_0+(2j+m)\theta)\rd t =: S_{1,j}+S_{2,j}+S_{3,j}.
\end{eqnarray*}
We have 
\begin{align}\label{PRF_1.1}
	&\sum_{j=0}^N C_{j,m}^2 S_{1,j} = \frac{c_2}{\theta}
	\sum_{j=0}^N C_{j,m}^2 \int^h_{-h}f(t+x_0+2(j+m)\theta)\rd t\nonumber\\
	&\;\; \leq c_5\sum_{j=0}^N \frac{j^{2m-2}}{\theta}\int^{x_0+2(j+m)\theta+h}_{x_0+2(j+m)\theta-h}\frac{t^{q}f(t)}{(x_0+2\theta j)^{q}}\rd t
	\leq\frac{c_6B h}{\theta^{q+1}}	\sum_{j=0}^N j^{2m-q-2}, 
\end{align}
where we have used that $C_{j,m}=\tbinom{j+m-1}{m-1}\leq c_0 j^{m-1}$, 
$f\in \sN_q(B)$ and $\theta>h$ for large $n$. The  term 
$\sum_{j=0}^N C_{j,m}^2 S_{2,j}$ is also bounded from above by the same expression as on the right hand side of 
(\ref{PRF_1.1}). Furthermore, 
\begin{align}\label{PRF_1.2}
	&\sum_{j=0}^N C_{j,m}^2 S_{3,j}= c_4\sum_{j=0}^N \frac{C_{j,m}^2}{\theta}\int^{m\theta}_{-m\theta}f(t+x_0+(2j+m)\theta) \rd t\nonumber\\	
	 &\;\;\leq \frac{c_8}{\theta} \sum_{j=0}^N j^{2m-2}\int^{x_0+2(j+m)\theta}_{x_0+2j\theta}\frac{t^{q}f(t)}{(x_0+2\theta j)^{q}}\rd t
	\leq \frac{c_9 B}{\theta^{q}}	\sum_{j=0}^N j^{2m-q-2}. 
\end{align}
Combining (\ref{PRF_1.2}), (\ref{PRF_1.1}) and (\ref{eq:var-1}) we conclude that
\begin{equation}\label{eq:var-2}
 {\rm var}_f\big[\hat{f}^+_{h, N}(x_0)\big]\;\leq\; \frac{c_{10}\theta^{2m-q} B\psi_N}{nh^{2m+1}},\;\;\;
 \;\;\psi_N:=\left\{\begin{array}{ll}
                            1, & q>2m-1,\\
                            \log N, & q=2m-1,\\
                            N^{2m-q-1}, & q<2m-1.
                           \end{array}\right.
\end{equation}
\par 
(b). Now we bound the bias of $\hat{f}^+_{h, N}(x_0)$. It is shown in \cite{BG:2019} that 
\[
 \rE_f \big[\hat{f}^+_{h, N}(x_0)\big]=\frac{1}{h}\int_{-\infty}^\infty K\Big(\frac{t-x_0}{h}\Big) f(t) \rd t + 
 T_N(f;x_0),
\]
where 
\[
 T_N(f; x_0)= \sum_{j=1}^m \tbinom{m}{j} \int_{-1}^1 K(y) f(yh+x_0+2\theta(N+1)j) \rd y.
\]
Taking into account that $f\in \sN_q(B)$ we obtain for any $j=1, \ldots, m$
\begin{align*}
\int_{-1}^1|K(y)|f(yh+x_0+2\theta(N+1)j) dy
\leq \frac{c_{11}}{h}\int_{x_0+2\theta(N+1)j-h}^{x_0+2\theta(N+1)j+h}f(y) dy
\\
 \leq \frac{c_{12} Bh}{h (x_0+2\theta N)^q}  \leq \frac{c_{13}B}{(\theta N)^q}.
\end{align*}	
This leads to the following upper bound on the bias of $\hat{f}_{h, N}(x_0)$:
\begin{align}\label{PRF_1.4}
\Big| \rE_f \big[\hat{f}^+_{h, N}(x_0)\big] - f(x_0)\Big| \leq c_{14} \Big(A h^{\alpha} + \frac{B}{\theta^q N^q}\Big). 	
\end{align}
\par 
(c). We complete the proof by combining the bounds in (\ref{eq:var-2})
and (\ref{PRF_1.4}) for  the cases $q>2m-1$, $q=2m-1$ and $q<2m-1$.
Sraightforward algebra shows that the following choice of $h=h_*$ and $N=N_*$ yields the theorem
result:
\begin{itemize}
 \item[(i)] if $q>2m-1$ then we set  
 \begin{equation}\label{eq:q>}
   h_*=c_1\Big(\frac{B}{A^2n}\Big)^{\frac{1}{2\alpha+2m+1}}, \;\;\; N_*\geq c_2 
 \Big( \frac{B^{\alpha+2m+1} n^\alpha}
 {A^{2m+1}}
 \Big)^{\frac{1}{q(2\alpha+2m+1)}};
 \end{equation}
 \item[(ii)] if $q=2m-1$ then 
\begin{equation}\label{eq:q=}
 h_*= c_3 \Big(\frac{B\log n}{A^2n}\Big)^{\frac{1}{2\alpha+2m+1}}, \;\;\;
  N_*= c_4\bigg\{\frac{B^{\alpha+2m+1}}{A^{2m+1}}\Big(\frac{n}{\log n}\Big)^\alpha \;
  \bigg\}^{\frac{1}{q(2\alpha+2m+1)}};
  \end{equation}
  \item[(iii)] if $q<2m-1$ then 
  \begin{equation}\label{eq:q<}
  h_*= c_5\Big(\frac{B^{(2m-1)/q}}{A^{(2m+q-1)/q}}\frac{1}{n}\Big)^{\frac{1}{2\alpha+2m+1+r}},
\;\;\;N_*= c_6(B/A)^{1/q} h_*^{-\alpha/q},
\end{equation}
\end{itemize}
where constants $c_1, \ldots c_6$ do not depend on $A$ and $B$.
\end{proof}
\subsection{Proof of Theorem~\ref{THM_2}}
\begin{proof}
Without loss of generality we fix $x_0$ to be $0$. The proof is split into a few steps: (i) defines two functions in $\mathscr{W}_{\alpha,q}(A,B)$ and provides their point-wise distance; (ii) bounds the $\chi^2$-divergence between densities of the observations; (iii) specifies the proper tuning parameters 
and provides the rate for the lower bound, and (iv) deals with 
derivation of the lower bound for the light tail regime.
\par 
(i).~For $s>1/2$ define
\begin{align}\label{PRF_2.1}
	f_0(x):=\frac{C(s)}{(1+x^2)^s},\;\;\; x\in\mathbb{R},
\end{align}
where $C(s)$ is a normalizing constant depending on $s$. Then, $f_0\in\mathscr{N}_q(B)$ for $1<q\le 2s$ since
$f_0(x)\le C(s)/x^{2s}\le B/x^q\text{ for }x>1$	with properly chosen $B>0$. In addition, since $f_0$ is infinitely differentiable, $f_0\in\mathscr{H}_\alpha(A)$ for any $\alpha$ with  properly chosen $A$. 
\par 
Define function $\eta_0$ on $\bR$ via its Fourier transform 
$\phi_{\eta_0}(\omega)=\int_{-\infty}^\infty \eta_0(x) e^{-i\omega x}\rd x$ as follows. 
Let  $\phi_{\eta_0}$ be an infinitely differentiable function on $\bR$ with the following properties: 
\begin{itemize}
\item[(a)] $\phi_{\eta_0}$ is supported on $[-1,1]$;
\item[(b)] $\phi_{\eta_0}$ is symmetric, $\phi_{\eta_0}(\omega)=\phi_{\eta_0}(-\omega)$, $\forall \omega\in \bR$; 
\item[(c)] given some fixed $\delta\in (0, 1/8)$,  $\phi_{\eta_0}(\omega)=1$ for $\omega \in [0,1-\delta)$, 
$\phi_{\eta_0}(\omega)=0$ for $\omega\geq 1$, and $\phi_{\eta_0}$ is monotone decreasing on $[1-\delta, 1)$. 
\end{itemize}
Given positive $h$ with $h<\pi/\theta$ and $N\in\mathbb{N}$, define
\begin{align}\label{PRF_2.2}
\phi_{\eta}(\omega):=&\sum^{2N}_{k=N+1}\left\{\phi_{\eta_0}\left(\frac{\omega-\pi k/\theta}{h}\right)+\phi_{\eta_0}\left(\frac{\omega+\pi k/\theta}{h}\right) \right\}.
\end{align}
Note that $\phi_\eta$ is  supported on: 
\begin{align}\label{PRF_2.4}
\bigcup^{2N}_{k=N+1}A_k(h),\;\;\;
A_k(h):=\left[\frac{-\pi k}{\theta}-h,
 \frac{-\pi k}{\theta}+h\right]\cup\left[\frac{\pi k}{\theta}-h,\frac{\pi k}{\theta}+h\right].
\end{align}
Then, define function $\eta$  through the inverse Fourier transform as follows:
\begin{align}\label{PRF_2.5}
\eta(x)=\frac{1}{2\pi} \int_{-\infty}^\infty \phi_{\eta}(\omega)e^{i\omega x}d\omega=2h\eta_0(hx)\sum^{2N}_{k=N+1}\cos\left(\frac{\pi kx}{\theta}\right)\text{ for }x\in\mathbb{R}.
\end{align}
In the subsequent proof the parameters $h$ and $N$ are specified so that 
$h\to 0$ and $N\to \infty$ as $n\to\infty$; thus, we tacitly assume that $N$ is large and $h$ is small 
for large enough 
sample size~$n$.
\par 
Given real numbers $M>0$ and $c_0>0$, define
\begin{align}\label{PRF_2.6}
f_1(x):=f_0(x)+c_0M\eta(x).	
\end{align}
We demonstrate that under appropriate choice of
$c_0$ and $M$. $f_1$ is a probability density from $\sW_{\alpha, q}(A, B)$ for any $h$ and $N$. 
Observe that $\phi_{\eta}(0)=0$ implies $\int^\infty_{-\infty}\eta(x)dx=0$ so that $f_1$ integrates to one. 
Moreover, since $\phi_{\eta_0}$ is infinitely differentiable and compactly supported, 
$\eta_0$ is a rapidly decreasing function, i.e., 
$|\eta_0^{(j)}(x) x^\ell | \leq c_{j,l}$ for any $j, \ell=0,1,2,\ldots$. 
In particular,  for some constant $c_1(s)$ depending on $s$ only one has  
$|\eta_0(x)|\le c_1(s)|x|^{-2s}$ for all $x\in \bR$. 
It follows from  (\ref{PRF_2.5}) that  $|\eta(x)|\le c_2 h^{-2s+1}|x|^{-2s}N$ for $x\in\mathbb{R}$. Therefore choosing 
\begin{equation*}
M=h^{2s-1}N^{-1} 
\end{equation*}
we obtain $c_0M|\eta(x)|\le f_0(x)$ for $c_0$ small enough. Therefore $f_0$ is non--negative, and it is 
a probability density. Moreover, $f_1\in\mathscr{N}_q(B)$ for $q\leq 2s$.
If $\alpha$ is  a positive integer then it follows from  (\ref{PRF_2.5}) that  
\begin{align*}
\left|\eta^{(\alpha)}(x)\right|=\left| 2h\sum_{i=0}^\alpha {\alpha \choose i}h^i\eta_0^{(i)}(xh)\sum_{k=N+1}^{2N}\cos^{(\alpha-i)}(\pi k x/\theta)	   \right|
\le c_2 h\sum_{i=0}^\alpha h^iN^{\alpha-i+1}\le c_3hN^{\alpha+1}. 
\end{align*}
Therefore, we can ensure $f_1\in\mathscr{H}_\alpha(A)$ by selecting $h$ and $N$ so that 
\begin{align}\label{PRF_2.7}
	MhN^{\alpha+1}=h^{2s}N^\alpha\le A.
\end{align}
Thus, 
under (\ref{PRF_2.7}) we have 
$f_0,f_1\in\mathscr{W}_{\alpha,q}(A,B)$. In addition, 
\begin{align}\label{PRF_2.8}
|f_1(0)-f_0(0)|=c_0M\eta(0)=c_0Mh\eta_0(0)N=c_4h^{2s}.
\end{align}
\par
(ii).~Now we derive an upper bound on the   $\chi^2$-divergence between the densities  of observations $f_{Y,0}=g\star f_0$ and $f_{Y,1}=g\star f_1$ that correspond to 
$f_0$ and $f_1$.  Observe the following expression: 
\begin{align*}
\chi^2(f_{Y,1},f_{Y,0}):=\int_{-\infty}^\infty \frac{(f_{Y,1}(x)-f_{Y,0}(x))^2}{f_{Y_0}(x)}dx\stackrel{(\ref{PRF_2.6})}{=}c_0^2M^2\int_{-\infty}^\infty \frac{|(g\star\eta)(x)|^2}{(g\star f_0)(x)}dx.
\end{align*}
Consider the denominator, $g\star f_0$, of the integrand. 
We have 
\begin{align*}
 (g\star f_0)(x)= C(s)\int^{\infty}_{-\infty}\frac{g(y)}{[1+(x-y)^2]^s}dy	\geq
 C(s)\int_{-\infty}^\infty \frac{g(y)}{2^s (1+y^2)^s(1+x^2)^s}\rd y\geq \frac{c_5}{(1+x^2)^s}, 
\end{align*}
where we have used the elementary inequality $1+|x-y|^2\leq 2(1+|x|^2)(1+|y|^2)$, $\forall x, y$.
Then the  $\chi^2$-divergence can be bounded:
\begin{align}\label{PRF_2.10}
\chi^2(f_{Y,1};f_{Y,0})\le c_6M^2\int^\infty_{-\infty}|(g\star\eta)(x)|^2dx+c_7M^2\int^\infty_{-\infty}x^{2s}|(g\star\eta)(x)|^2dx.
\end{align}
\par
Let us handle the second integral on the right-hand side. 
For any positive integer number $s$ we have 
\begin{align}\label{PRF_2.11}
\int^\infty_{-\infty}x^{2s}|(g\star\eta)(x)|^2dx	=\frac{1}{2\pi}\int^\infty_{-\infty}\left|\frac{d^s}{d\omega^s}\phi_g(\omega)\phi_{\eta}(\omega)\right|^2d\omega.
\end{align}
Note that 
\begin{align*}
\frac{d^s}{d\omega^s}\phi_g(\omega)\phi_{\eta}(\omega)=&\sum_{j=0}^s{s \choose j}\phi_g^{(j)}(\omega)\phi_{\eta}^{(s-j)}(\omega)\\
=&	\sum_{j=0}^s{s \choose j}\frac{\phi_g^{(j)}(\omega)}{h^{s-j}}\sum_{k=N+1}^{2N}\left\{\phi_{\eta_0}^{(s-j)}\left(\frac{\omega-\pi k/\theta}{h}\right) + \phi_{\eta_0}^{(s-j)}\left(\frac{\omega+\pi k/\theta}{h}\right)\right\}.
\end{align*}
Furthermore, $\phi_g^{(j)}$ can be expanded by Fa\'{a} di Bruno formula for $j\in\mathbb{N}$: 
if $\phi_{g_0}(\omega):= \sin(\theta \omega)/(\theta\omega)$ then $\phi_g(\omega)=[\phi_{g_0}(\omega)]^m$ and 
\begin{align*}
\phi_g^{(j)}(\omega)=\frac{d^j}{d\omega^j}\left(\frac{\sin\theta\omega}{\theta\omega}\right)^m
=\sum_{l=1}^j j\cdots (j-l+1)\left(\frac{\sin\theta\omega}{\theta\omega} \right)^{m-l}B_{j,l}\Big(
\phi_{g_0}^\prime(\omega),\ldots,\phi_{g_0}^{(j-l+1)}(\omega) \Big),
\end{align*}
where $B_{j,l}$  denotes the Bell polynomials.  
Recall that $B_{j,l}$ is a homogeneous polynomial in 
$j$ variables of degree $l$, and note that $|\phi_{g_0}^{(j)}(\omega)|\le c_{8}(|\omega|^{-1}\wedge 1)$, $\forall j$. 
Then,
\begin{align}\label{eq:phi-g-j}
\left|\phi_{g}^{(j)}(\omega)\right|\le 
c_{9}\sum_{l=1}^j\left|\frac{\sin\theta\omega}{\theta\omega}\right|^{m-l}|\theta\omega|^{-l}
=\frac{c_{9}}{|\theta\omega|^m}\sum_{l=1}^j|\sin\theta\omega|^{m-l}.	
\end{align}
Combining the above results and the fact that sets $A_k(h)$ in (\ref{PRF_2.2}) are disjoint for  $k=N+1,\ldots,2N$, 
we bound the integral in (\ref{PRF_2.11}) as follows:
\begin{align*}
	&\int^\infty_{-\infty}\left|\sum_{j=0}^s{s \choose j}\frac{\phi_g^{(j)}(\omega)}{h^{s-j}}\sum_{k=N+1}^{2N}\left\{\phi_{\eta_0}^{(s-j)}\left(\frac{\omega-\pi k/\theta}{h}\right) + \phi_{\eta_0}^{(s-j)}\left(\frac{\omega+\pi k/\theta}{h}\right)\right\} \right|^2 \rd \omega\\
	\le& c_{10}h^{-2s}\sum_{k=N+1}^{2N}\int_{A_k(h)}\left|\sum_{j=0}^s h^j\phi_g^{(j)}(\omega)\right|^2
	\rd\omega\\
	\le& c_{11}h^{-2s} \sum_{k=N+1}^{2N}\int_{A_k(h)}\left(\left|\frac{\sin\theta\omega}{\theta\omega}\right|^{2m}+\frac{1}{|\theta\omega|^{2m}}\sum_{j=1}^sh^{2j}\sum_{l=1}^j|\sin\theta\omega|^{2m-2l} \right)\rd\omega\\
	\le& c_{12}h^{2m+1-2s}\sum_{k=N+1}^{2N}\frac{1}{k^{2m}}=c_{13}h^{2m-2s+1}N^{-2m+1}.
\end{align*}
In addition, the first integral on the left-hand side in (\ref{PRF_2.10}) can be bounded with $s=0$, so that
\begin{align*}
\int^\infty_{-\infty}|(g\star\eta)(x)|^2 \rd x\le 	c_{14}h^{2m+1}N^{-2m+1}.
\end{align*}
Therefore, for positive integer $s$,
\begin{align}\label{PRF_2.13}
\chi^2(f_{Y,1};f_{Y,0})\le c_{14}M^2h^{2m+1}N^{-2m+1}+ c_{13}M^2h^{2m-2s+1}N^{-2m+1}\nonumber\\
\le c_{15}h^{2m+2s-1}N^{-2m-1}.
\end{align}
The same upper bound holds for any non-integer $s\geq 0$; this fact  is due to the interpolation inequality for the Sobolev spaces, see, e.g., \citet{Au:2011} for the details. 
\par 
(iii).~Now, based on (\ref{PRF_2.7}) and (\ref{PRF_2.13}), we  specify parameters $h=h_*$ and 
$N=N_*$ as follows:
\begin{align*}
N_*:=\left(\frac{A}{h_*^{2s}}\right)^{1/\alpha},\;\;\;\;\;\;
h_*:=\left(\frac{A^\frac{2m+1}{\alpha}}{n}\right)^\frac{\alpha}{(2m+2s-1)\alpha+2s(2m+1)}. 	
\end{align*}
Under this choice (\ref{PRF_2.7}) holds, and  $\chi^2(f_{Y,1},f_{Y,0})\le c_{15}/n$.
Then the lower bound on the minimax risk is  obtained by plugging these expressions 
in  (\ref{PRF_2.8}) 
and letting $2s=q>1$:
\begin{align}\label{eq:LB-1}
\mathcal{R}^*_n[\mathscr{W}_{\alpha,q}(A,B)]\ge c_4
\left(\frac{A^\frac{2m+1}{\alpha}}{n}\right)^\frac{\alpha}{2\alpha+2m+1 + (\alpha/q)(2m-1-q)}.
\end{align}
\par 
(iv).~To complete 
the proof of the theorem it remains to observe that in the considered problem the following  standard lower bound 
on the minimax risk
can be also established:
\begin{align}\label{eq:LB-2}
\mathcal{R}^*_n[\mathscr{W}_{\alpha,q}(A,B)]\ge c_4
\left(\frac{A^\frac{2m+1}{\alpha}}{n}\right)^\frac{\alpha}{2\alpha+2m+1}.
\end{align}
For completeness, we provide the proof sketch.  Let $f_0$ be given by 
(\ref{PRF_2.1}), and let $\eta$ be the function defined via its Fourier transform $\phi_\eta$ as follows
\[
 \phi_\eta(\omega)= \phi_{\eta_0}(2\omega h-3) + \phi_{\eta_0}(2\omega h +3),
\]
where $\phi_{\eta_0}$ is a function with properties (a)--(c). Obviously, $\phi_\eta$ is symmetric, supported 
on $[-2/h,-1/h]\cup [1/h,2/h]$, and 
\[
\eta(x)= \frac{1}{2\pi} \int_{-\infty}^\infty \big[\phi_{\eta_0}(2\omega h-3) + \phi_{\eta_0}(2\omega h+3)\big]
e^{i\omega x}
\rd \omega = \frac{2}{h} \eta_0\Big(\frac{x}{2h}\Big)\cos\Big(\frac{3x}{2}\Big).
\]
The function $f_1$ is defined by 
(\ref{PRF_2.6}), and the choice $M=Ah^{\alpha+1}$ and properties of function $\eta_0$ guarantee that
$f_1$ is a density from the class $\sW_{\alpha, q}(A, B)$ with $q\leq 2s$. With this construction 
$|f_0(0)-f_1(0)|=c_0M\eta(0)=c_{16} Ah^\alpha$.
The upper bound on the 
$\chi^2$--divergence between $f_{Y,0}$ and $f_{Y, 1}$ is  computed along the same lines as above with the following 
modifications. Now we apply (\ref{eq:phi-g-j}) to get
\begin{align*}
\Big|\frac{d^s}{d\omega^s}\phi_g(\omega)\phi_{\eta}(\omega) \Big|\leq &
\sum_{j=0}^s {s \choose j}\Big|\phi_g^{(j)}(\omega)\phi_{\eta}^{(s-j)}(\omega)\Big| \leq 	
c_{17} |\theta \omega|^{-m} \sum_{j=0}^s |\phi_{\eta}^{(s-j)}(\omega)|,
\end{align*}
and, by properties of function $\phi_\eta$,
\[
\int^\infty_{-\infty}x^{2s}|(g\star\eta)(x)|^2\rd x \leq c_{18}\int_{1/h}^{2/h} |\omega|^{-2m}
\rd \omega =c_{19}h^{2m-1}. 
\]
The same upper bound holds for the integral $\int^\infty_{-\infty}|(g\star\eta)(x)|^2\rd x$ which leads to 
\[
\chi^2(f_{Y,1}; f_{Y,0})\leq c_{20} M^2 h^{2m-1}= c_{20} A^2 h^{2\alpha+2m+1}. 
\]
Then (\ref{eq:LB-2}) follows from the choice $h_*= (A^2 n)^{-1/(2\alpha+2m+1)}$.
\par 
Combining (\ref{eq:LB-1}) and (\ref{eq:LB-2}) and noting that 
the following relation holds  for $1<q<2m-1$ 
\begin{align*}
\frac{\alpha}{2\alpha+2m+1 + (\alpha/q)(2m-1-q)} \le \frac{\alpha}{2\alpha+2m+1},
\end{align*}
we complete the proof. 
\end{proof}

\subsection{Proof of Corollary~\ref{THM_3}}
\begin{proof} The upper bound (\ref{eq:upper-holder}) is obtained directly from Theorem~\ref{THM_1}
applied with $q=1$. We need to establish (\ref{eq:lower-holder}) only. 
The proof goes along the lines of the proof of
Theorem~\ref{THM_2} with minor modifications that are indicated below.
\par 
Define 
\begin{align*}
f_0(x):=\frac{h}{\pi(1+h^2x^2)},\;\;\;x\in\mathbb{R},
\end{align*}
where $h>0$ is a parameter to be specified.
Obviously, $f_0\in\mathscr{H}_\alpha(A)$ for small enough $h$. 
Using the function $\eta$ defined in (\ref{PRF_2.2}), (\ref{PRF_2.4}), and (\ref{PRF_2.5}), let
\begin{align*}
	f_1(x):=f_0(x)+c_0M\eta(x)\text{ for }x\in\mathbb{R}.	
\end{align*}
Similarly to the proof of Theorem~\ref{THM_2}, $|\eta(x)|\leq c_1h^{-1}N|x|^{-2}$.
Set $M:=N^{-1}$, so that $c_0M|\eta(x)|=c_0c_1/(h|x|^2)\le f_0(x)$ holds for sufficiently small $c_0$.
Since we use the same function $\eta$ in Theorem~\ref{THM_2}, we can ensure $f_1\in\mathscr{H}_\alpha(A)$ by setting 
\begin{align}\label{PRF_3.1}
MhN^{\alpha+1}=hN^\alpha\le A.
\end{align}
Therefore, for $x_0=0$, we have the following point-wise distance
\begin{align*}
|f_1(0)-f_0(0)|=c_1M\eta(0)=c_1Mh\eta_0(0)N=c_2h.
\end{align*}
The bound on the $\chi^2$--divergence takes the following form   
\begin{align}\label{PRF_3.2}
\chi^2(f_{Y,1};f_{Y,0})\le& c_3\frac{M^2}{h}\int^\infty_{-\infty}|(g\star\eta)(x)|^2dx+c_4M^2h\int^\infty_{-\infty}x^{2}|(g\star\eta)(x)|^2dx\nonumber\\
\le& c_5(M^2/h)h^{2m+1}N^{-2m+1}+ c_6(M^2h)h^{2m-1}N^{-2m+1}
\le c_7h^{2m}N^{-2m-1}.
\end{align}
Based on (\ref{PRF_3.1}) and (\ref{PRF_3.2}), we choose $h=h_*$ and $N=N_*$ as follows:
\begin{align*}
N_*:=(A/h_*)^{1/\alpha},\;\;\;	h_*:=(A^{2m+1}/n^\alpha)^\frac{1}{2m\alpha+2m+1}
\end{align*}
which leads to the announced result. 
\end{proof}

\subsection{Proof of Theorem~\ref{th:oracle-ineq}}
\begin{proof}
(I). The error of estimator $\hat{f}^+_\tau(x_0)$ is 
\[
 |\hat{f}^+_\tau (x_0)- f(x_0)|\leq |B_\tau(x_0; f)| + |\xi_\tau (x_0)|,
\]
where  $B_\tau(x_0; f)$ is the bias term, and $\xi_\tau (x_0)$ is the stochastic error given by  (\ref{eq:xi}).
The bias term is expressed as follows (see the proof of Theorem~\ref{THM_1}):
\begin{multline}
  B_\tau(x_0; f) := 
 \rE_f \big[\hat{f}^+_{h, N}(x_0)\big]- f(x_0) 
 \\
 = \frac{1}{h}\int_{-\infty}^\infty K\Big(\frac{t-x_0}{h}\Big) [f(t)-f(x_0)] \rd t 
+ 
 \sum_{j=1}^m \tbinom{m}{j} (-1)^j\int_{-1}^1 K(y) f(yh+x_0+2\theta(N+1)j) \rd y
 \\
 = \sum_{j=0}^m \tbinom{m}{j} (-1)^j\int_{-1}^1
 K(y)\Big[f(yh+x_0+2\theta (N+1)j)- f(x_0+2\theta(N+1)j)\Big] \rd y
 \\
 +\sum_{j=1}^m \tbinom{m}{j} (-1)^jf(x_0+2\theta(N+1)j).
\nonumber
\end{multline}
Therefore by definitions of $\bar{B}_h(f)$ and $\bar{B}_N(x_0;f)$  [see (\ref{eq:barB-h}), (\ref{eq:barB-N})]
we have
\[
 |B_\tau(x_0;f)| \leq 2^m \bar{B}_h(f) + 2^m \bar{B}_N(x_0; f) \leq \bar{B}_\tau(x_0;f),
\]
where $\bar{B}_\tau(x_0;f)$ is defined in (\ref{eq:bias-UB}).
\par 
(II). 
Now we demonstrate that  
\[ 
|B_{\tau \vw \tau^\prime}(x_0;f) - B_{\tau^\prime}(x_0;f)|\leq \bar{B}_\tau(x_0;f),\;\;\;\forall 
\tau, \tau^\prime \in \cT.
\]
For this purpose  denote 
\begin{align*}
 S_h(x):=\frac{1}{h}\int_{-\infty}^\infty K\Big(\frac{t-x}{h}\Big) [f(t)-f(x)] \rd t
 \\
 T_N(x):= \sum_{j=1}^m \tbinom{m}{j} (-1)^j f(x+2\theta(N+1)j)
\end{align*}
and write 
\begin{equation}\label{eq:bias-decomposition}
 B_\tau(x_0; f) =  S_h(x_0)  + T_N(x_0)+\sum_{j=1}^m \tbinom{m}{j} (-1)^j S_h(x_0+2\theta(N+1)j).
\end{equation} 
In view of (\ref{eq:bias-decomposition}) for any pair $\tau=(h, N)$, $\tau^\prime=(h^\prime, N^\prime)$ 
we have 
\begin{align}
& B_{\tau \vw \tau^\prime}(x_0;f) - B_{\tau^\prime}(x_0;f)  = 
 \big[S_{h\vee h^\prime} (x_0) - S_{h^\prime}(x_0)\big]  + \big[T_{N\wedge N^\prime}(x_0) - 
 T_{N^\prime} (x_0)\big]
\nonumber
 \\
&\;\;\;  + 
 \sum_{j=1}^m \tbinom{m}{j}(-1)^j \Big[ S_{h\vee h^\prime} (x_0+2\theta(N\wedge N^\prime +1)j) -
 S_{h^\prime} (x_0+2\theta(N^\prime +1)j)\Big].
\label{eq:B-tau}
 \end{align}
We consider the three terms on the right hand side of (\ref{eq:B-tau}):
\begin{align}
 \sup_{h^\prime \in \cH} \big|S_{h\vee h^\prime} (x_0) - S_{h^\prime}(x_0)\big|=
 \sup_{h^\prime\leq h} \big[S_{h\vee h^\prime} (x_0) - S_{h^\prime}(x_0)\big| 
 \nonumber
 \\ 
 \leq 
 \big|S_h(x_0)\big| + \sup_{h^\prime \leq h}\big|S_{h^\prime}(x_0)\big| \leq 
 2\sup_{h^\prime \leq h}\big|S_{h^\prime}(x_0)\big|,
\label{eq:B-1}
 \end{align}
 and similarly
\begin{equation}\label{eq:B-2}
\sup_{N^\prime \in \cN}\big|T_{N\wedge N^\prime}(x_0) - T_{N^\prime} (x_0)\big|
 \leq 2 \sup_{N^\prime\geq  N}\sum_{j=1}^m \tbinom{m}{j}  f(x_0+2\theta (N^\prime +1)j).
\end{equation}
Furthermore 
\begin{align}
&\sup_{h^\prime, N^\prime}  \big|S_{h\vee h^\prime} (x_0+2\theta(N\wedge N^\prime +1)j) -
 S_{h^\prime} (x_0+2\theta(N^\prime +1)j)\big|
\nonumber
 \\
&\;\leq \sup_{h^\prime, N^\prime} \big|S_{h\vee h^\prime} (x_0+2\theta(N\wedge N^\prime +1)j) -
 S_{h^\prime} (x_0+2\theta(N^\prime\wedge N +1)j)\big|
\nonumber
 \\
&\;\;\;\;\;\;\;+
 \sup_{h^\prime, N^\prime} \big|S_{h^\prime} (x_0+2\theta(N\wedge N^\prime +1)j) -
 S_{h^\prime} (x_0+2\theta(N^\prime +1)j)\big|
\nonumber
 \\
&\;\leq  2\sup_{h^\prime \leq h} \big\| S_{h^\prime}\big\|_\infty + 
 2 \sup_{h^\prime\in \cH} \sup_{N^\prime\geq N} \big| S_{h^\prime} (x_0+2\theta(N^\prime +1)j)\big|
\;\leq\; 2\sup_{h^\prime \leq h} \big\| S_{h^\prime}\big\|_\infty 
\nonumber
\\
&\;\;\;\;\;+ 
2\sup_{h\in \cH}\sup_{N^\prime\geq N}
\bigg|\int_{-1}^1 K(y) f(yh +x_0+2\theta(N^\prime+1)j) \rd y\bigg|
 + 2\sup_{N^\prime\geq  N} 
f(x_0+2\theta(N^\prime +1)j)
\nonumber
\\ 
&\;\leq  2\sup_{h^\prime \leq h} \big\| S_{h^\prime}\big\|_\infty  + 2 (1+\|K\|_1) \sup_{|t|\leq \theta}
\sup_{N^\prime \geq N} f(t+x_0+2\theta (N^\prime+1)j).
 \label{eq:B-3}
 \end{align}
Combining (\ref{eq:B-1})--(\ref{eq:B-3}) with (\ref{eq:B-tau}) we obtain 
\begin{align}
 &\sup_{\tau^\prime\in \cT}  \big|B_{\tau\vw\tau^\prime}(x_0;f)- B_{\tau^\prime}(x_0; f)\big| 
\nonumber
 \\
 &\;\;\leq  \;2^{m+1} \sup_{h^\prime \leq h}  \big\|S_{h^\prime}\big\|_\infty 
 + 2^{m+1}(1+\|K\|_1) \max_{1\leq j\leq m} \sup_{|t|\leq \theta}
\sup_{N^\prime \geq N} f(t+x_0+2\theta (N^\prime+1)j) 
\nonumber
\\
&\;\;= 2^{m+1} \bar{B}_h(f) +
2^{m+1}(1+\|K\|_1) \bar{B}_N(x_0; f) \leq \bar{B}_\tau(x_0;f),
\label{eq:bar-B}
\end{align}
where $\bar{B}_h(f)$, $\bar{B}_N(x_0; f)$ and $\bar{B}_\tau(x_0;f)$ are defined in (\ref{eq:barB-h}),
(\ref{eq:barB-N}), and  (\ref{eq:bias-UB}) respectively.
\par 
(III).  
Let $\hat{\tau}=(\hat{h}, \hat{N})$ be the parameter selected by the rule 
(\ref{EQN_4.4})--(\ref{EQN_4.5}). For any $\tau\in \cT$ we have by the triangle inequality
\begin{equation}\label{eq:decomposition}
 |\hat{f}_{\hat{\tau}}^+(x_0) - f(x_0)| \leq |\hat{f}^+_{\hat{\tau}} (x_0)- 
 \hat{f}^+_{\hat{\tau} \vw \tau}(x_0)| +  |\hat{f}^+_{\tau\vw \hat{\tau}}(x_0) - \hat{f}_{\tau}^+(x_0)|+
 |\hat{f}_\tau^+(x_0)- f(x_0)|.
\end{equation}
Now we bound the terms on the right hand side separately. 
\par 
We begin with the following simple observation: it follows from (\ref{EQN_4.4}) that 
\begin{align*}
& \hat{R}_\tau(x_0) - \hat{\Lambda}_\tau(\kappa)
- \sup_{\tau^\prime\in \cT}\hat{\Lambda}_{\tau\vw \tau^\prime}(\kappa) =
\sup_{\tau^\prime \in \cT}
\Big[ \big|\hat{f}^+_{\tau\vw\tau^\prime}(x_0)- \hat{f}_{\tau^\prime}^+(x_0)\big| - \hat{\Lambda}_{\tau\vw\tau^\prime}(\kappa)-
\hat{\Lambda}_{\tau^\prime}(\kappa)\Big]_+
\\
&\;\;\leq \sup_{\tau^\prime\in \cT} \big|B_{\tau \vw\tau^\prime}(x_0; f)- B_{\tau^\prime}(x_0;f)\big| 
+ \sup_{\tau^\prime\in \cT}\Big[ |\xi_{\tau\vw\tau^\prime}(x_0)- \xi_{\tau^\prime}(x_0)\big| -
\hat{\Lambda}_{\tau\vw\tau^\prime}(\kappa) - \hat{\Lambda}_{\tau^\prime}(\kappa)\Big]_+.
\end{align*}
Hence by (\ref{eq:bar-B}) 
\begin{equation}\label{eq:R-hat}
 \hat{R}_\tau(x_0) \leq \bar{B}_\tau (x_0; f) + 2 \hat{\zeta}(x_0) + \hat{\Lambda}_\tau(\kappa)
+ \sup_{\tau^\prime\in \cT}\hat{\Lambda}_{\tau\vw \tau^\prime}(\kappa),
\end{equation}
where 
\[ 
\hat{\zeta}(x_0):= \sup_{\tau\in \cT} \big[|\xi_\tau(x_0)| - \hat{\Lambda}_\tau(\kappa)\big]_+~. 
\]
Therefore for any $\tau, \tau^\prime\in \cT$
\begin{align*}
 &\big|\hat{f}^+_{\tau\vw\tau^\prime}(x_0)- \hat{f}^+_{\tau^\prime}(x_0)\big| \leq 
 \big|B_{\tau\vw\tau^\prime}(x_0;f)- B_{\tau^\prime}(x_0;f)\big| + \big|\xi_{\tau\vw \tau^\prime}(x_0)-
 \xi_{\tau^\prime}(x_0) \big|
 \\
 &\;\;\;\;\;\;
 \leq  \bar{B}_\tau(x_0;f) + 2\hat{\zeta}(x_0) + \hat{\Lambda}_{\tau\vw\tau^\prime}(\kappa) + 
 \hat{\Lambda}_{\tau^\prime}(\kappa) \leq \bar{B}_{\tau}(x_0;f)+2\hat{\zeta}(x_0)+ 
 \hat{R}_{\tau^\prime}(x_0),
\end{align*}
where the last inequality follows from the definition of $\hat{R}_\tau(x_0)$.
This inequality together with (\ref{eq:R-hat}) imply the following bound on the first term on the right hand side of (\ref{eq:decomposition}):
\begin{align}
 |\hat{f}^+_{\hat{\tau} \vw \tau}(x_0)-\hat{f}^+_{\hat{\tau}} (x_0)| \leq \bar{B}_\tau(x_0;f)
 +2 \hat{\zeta}(x_0) + \hat{R}_{\hat{\tau}}(x_0) \leq \bar{B}_\tau(x_0;f)
 +2 \hat{\zeta}(x_0) + \hat{R}_{\tau}(x_0)
 \nonumber
 \\
 \leq 2\bar{B}_\tau(x_0; f) + 4\hat{\zeta}(x_0)+ \hat{\Lambda}_\tau(\kappa)+ \sup_{\tau^\prime}
 \hat{\Lambda}_{\tau\vw \tau^\prime}(\kappa),
 \label{eq:first-term}
\end{align}
where in the penultimate inequality we have used that $\hat{R}_{\hat{\tau}}(x_0)\leq \hat{R}_\tau(x_0)$
for any $\tau\in \cT$.
\par 
We proceed with bounding the second term on the right hand side of (\ref{eq:decomposition}):
by definition of $\hat{R}_{\hat{\tau}}(x_0)$ we have 
\begin{align}
 |\hat{f}^+_{\tau\vw \hat{\tau}}(x_0)- \hat{f}^+_{\tau}(x_0)| \pm [\hat{\Lambda}_{\tau\vw\hat{\tau}}(\kappa)
 + \hat{\Lambda}_{\tau}(\kappa)] \leq \hat{R}_{\hat{\tau}}(x_0) + \sup_{\tau^\prime\in \cT}
 \hat{\Lambda}_{\tau\vw \tau^\prime}(\kappa)
 + \hat{\Lambda}_{\tau}(\kappa)
 \nonumber 
 \\
 \leq \hat{R}_\tau(x_0)+\sup_{\tau^\prime\in \cT}
 \hat{\Lambda}_{\tau\vw \tau^\prime}(\kappa)
 + \hat{\Lambda}_{\tau}(\kappa) 
 \nonumber
 \\
 \leq 
 \bar{B}_\tau(x_0; f)+ 2\hat{\zeta}(x_0) + 2\sup_{\tau^\prime\in \cT}
 \hat{\Lambda}_{\tau\vw \tau^\prime}(\kappa)
 + 2\hat{\Lambda}_{\tau}(\kappa).
 \label{eq:second-term}
\end{align}
\par 
Finally 
\begin{equation}\label{eq:third-term}
 |\hat{f}_\tau^+(x_0) - f(x_0)| \leq |B_\tau(x_0; f)| + |\xi_\tau (x_0)|
 \leq \bar{B}_\tau(x_0;f) + \Lambda_\tau(\kappa) + \zeta(x_0),
\end{equation}
where we recall that 
\[
\zeta(x_0):=\sup_{\tau\in \cT} \Big[|\xi_\tau(x_0)| - \Lambda_\tau(\kappa)\Big]_+.
\]
Combining (\ref{eq:first-term}), (\ref{eq:second-term}), (\ref{eq:third-term})
and (\ref{eq:decomposition}) we obtain 
\begin{align*}
 \big|\hat{f}^+_{\hat{\tau}}(x_0)- f(x_0)\big| \leq \inf_{\tau\in \cT}
 \Big\{ 4\bar{B}_\tau(x_0;f) + 3\hat{\Lambda}_\tau(\kappa) + 3 \sup_{\tau^\prime\in \cT}
 \hat{\Lambda}_{\tau \vw \tau^\prime}(\kappa) + \Lambda_\tau(\kappa) \Big\}
 \\
 \;+\; 6\hat{\zeta}(x_0)+ \zeta(x_0).
\end{align*}
\par 
(IV). We complete the proof using Lemmas~\ref{lem:Lambda-Lambda-tilde} and~\ref{lem:xi} in Appendix.
Observing that $\hat{\Lambda}_\tau(\kappa)=7\tilde{\Lambda}_\tau(\kappa)$ and applying the first inequality in 
(\ref{eq:LL}) we have 
\begin{align*}
 \hat{\zeta}(x_0) \leq \zeta(x_0) + \sup_{\tau\in \cT} \big[\Lambda_\tau(\kappa)- 7\tilde{\Lambda}_\tau(\kappa)\big]_+ \leq \zeta(x_0) +2c\eta(x_0),
\end{align*}
where $c=2^{-m-2}\theta \|K^{(m)}\|_\infty^{-1}$ [cf.~Lemma~\ref{lem:Lambda-Lambda-tilde}].
Then using the second inequality in (\ref{eq:LL}) in order to bound $\hat{\Lambda}_\tau(\kappa)$ and 
$\sup_{\tau^\prime\in \cT} \hat{\Lambda}_{\tau \vw \tau^\prime}(\kappa)$ in terms of $\Lambda_\tau(\kappa)$ 
we obtain
\begin{align*}
 \big|\hat{f}^+_{\hat{\tau}}(x_0)- f(x_0)\big| \leq \inf_{\tau\in \cT}
 \Big\{ 4\bar{B}_\tau(x_0;f) + 127 \Lambda_\tau(\kappa) + 126 \sup_{\tau^\prime\in \cT}
 \Lambda_{\tau \vw \tau^\prime}(\kappa) \Big\}
 \\
 \;+ 7\zeta(x_0) + (42+12c)\eta(x_0) + \frac{42\kappa}{n}.
\end{align*}
By definition of the opeartion $\vw$ and by definition of $\sigma_\tau^2$ and $u_\tau$ 
[see  (\ref{eq:sigma-tau}) and 
(\ref{eq:eta-u})] we have that $\sigma^2_{\tau\vw\tau^\prime} \leq \sigma^2_\tau$ and 
$u_{\tau\vw\tau^\prime} \leq u_\tau$ for any $\tau, \tau^\prime \in \cT$; therefore
$\sup_{\tau^\prime\in \cT} \Lambda_{\tau\vw \tau^\prime}(\kappa) \leq 
 \Lambda_\tau(\kappa)$ for all $\tau\in \cT$.
We complete the proof by setting $\delta(x_0)=\zeta(x_0)+\eta(x_0)$ and using Lemma~\ref{lem:xi}. 
\end{proof}

\subsection{Proof of Corollary~\ref{cor:2}}
\begin{proof}
Below  $c_1, c_2, \ldots$
stand for positive constants independent of $n$, $A$ and $B$.
The proof goes along the following lines. We select values of $h$ and $N$ from $\cH\times \cN$ and apply 
the oracle inequality of Theorem~\ref{th:oracle-ineq}.
\par 
The proof of Theorem~\ref{THM_1} shows that if $f\in \sW_{\alpha, q}(A, B)$ then 
 \[
  \bar{B}_h(f) \leq c_1 Ah^{\alpha},\;\;\;\bar{B}_N(x_0;f)\leq c_2 B\theta^{-q} N^{-q}.
 \]
 Furthermore, by (\ref{eq:var-2})
\begin{equation*} 
 \sigma_\tau^2 \;\leq\; \frac{c_{3}\theta^{2m-q} B\psi_N}{h^{2m+1}},\;\;\;
 \;\;\psi_N:=\left\{\begin{array}{ll}
                            1, & q>2m-1,\\
                            \log N, & q=2m-1,\\
                            N^{2m-q-1}, & q<2m-1.
                           \end{array}\right.
\end{equation*}
In addition, with $\kappa_*=\kappa_0\log n$
we have
\begin{equation}\label{eq:Lambda*}
 \Lambda_\tau(\kappa_*) \leq c_5 \bigg( \frac{B^{1/2}\psi_N^{1/2}}{h^{m+1/2}} \sqrt{\frac{\kappa_0\log n}{n}} + 
\frac{N^{m-1}}{h^{m+1}} \frac{\kappa_0\log n}{n} \bigg).
\end{equation}
First we note that for all $h_{\min}\leq h\leq h_{\max}$ and $N\leq N_{\max}$
and all sufficiently large $n$
\[
 \Lambda_\tau(\kappa_*) \leq c_6 \frac{B^{1/2}\psi_N^{1/2}}{h^{m+1/2}}\sqrt{\frac{\kappa_0\log n}{n}}.
\]
Indeed, this inequality follows from (\ref{eq:Lambda*})   because by  the choice of 
 $h_{\min}$ and 
$N_{\max}$ for large $n$ one has 
\[
 h_{\min} \Big(\frac{n}{\log n}\Big) =\Big(\frac{n}{\log n}\Big)^{2m/(2m+1)} \geq N_{\max}^{2m-2} =
 \Big(\frac{n}{\log n}\Big)^{(2m-2)/(2m)}.
\]
Thus, using  (\ref{eq:oracle}) we have 
\[
 |\hat{f}_*(x_0) - f(x_0)|\leq c_7\inf_{(h, N)\in \cH\times \cN} 
 \bigg\{Ah^{\alpha} + \frac{B}{\theta^{q} N^{q}} + \frac{B^{1/2}\psi_N^{1/2}}{h^{m+1/2}}
 \sqrt{\frac{\log n}{n}}\,\bigg\} + c_8 \bigg(\delta(x_0)+ \frac{\kappa_0\log n}{n}\bigg).
\]
\par 
Now we set  $h_*$ and $N_*$ to be defined by formulas (\ref{eq:q<}), (\ref{eq:q=}) and (\ref{eq:q>}) with 
$n$ replaced by~$n/\log n$. Note that these values of $h$ and $N$ balance the bias and stochastic error
bounds on the right hand side of the previous display formula [for details see the proof of Theorem~\ref{THM_1}].
We need to verify that $h_*$ and $N_*$ satisfy $h_*\geq h_{\min}$ and $N_*\leq N_{\max}$ for large $n$. 
The first inequality is evident because 
$1/(2\alpha+2m+1)\geq 1/(2m+1)$ for all $\alpha>0$. To check the inequality $N_*\leq N_{\max}$
we note that 
$N_*= O\big((n/\log n)^{\frac{\alpha}{q(2m+2\alpha+1+r)}}\big)$ in the case $1\leq q<2m-1$
and 
\[
 \frac{\alpha}{q(2m+2\alpha+1+r)}=
 \frac{\alpha}{\alpha (2m-1+q)+q(2m+1)}\leq \frac{1}{2m}
\]
for all $\alpha>0$. If $q>2m-1$ then $N_*= O\big((n/\log n)^{\frac{\alpha}{q(2m+2\alpha+1)}}\big)$, and 
\[
 \frac{\alpha}{q(2m+2\alpha+1)} \leq \frac{\alpha}{(2m-1)(2m+2\alpha+1)}\leq \frac{1}{4m-2},\;\;\;\forall 
 \alpha >0.
\]
Thus, we always have $N_*\leq N_{\max}$ for large $n$.  
The inequalities   $h_*\geq h_{\min}$ and 
$N_*\leq N_{\max}$ imply that sets $\cH$ and $\cN$ contain  
elements that bound  $h_*$ and $N_*$ from below and from above within constant factors. This  yields
\[
 |\hat{f}_*(x_0) - f(x_0)|\leq c_9 \varphi(n/\log n) + c_8 \bigg(\delta(x_0)+ \frac{\kappa_0\log n}{n}\bigg),
\]
where function $\varphi(\cdot)$ is defined in (\ref{eq:varphi}). 
\par 
To complete the proof we note that $M_h=O(\log_2 n)$, $M_N=O(n^{1/(2m)})$, and 
\[
 \bar{\Lambda}(\kappa_*)\leq  c_{10} \frac{N_{\max}^{m-1}}{h_{\min}^{m+1/2}}\sqrt{\frac{\log n}{n}}
 \bigg(1+\frac{N_{\max}^{m-1}}{h_{\min}^{m+1}}\bigg) \leq c_{11}\Big(\frac{ n}{\log n}\Big)^{3/2},
\]
so that if $\kappa_0\geq 5$ then in view of (\ref{eq:delta}) for large $n$
\[
 \rE_f [\delta(x_0)]^2 \leq c_{12} (\log_2 n)n^{1/2m} \Big(\frac{n}{\log n}\Big)^3 e^{-\kappa_0\log n}
 \leq c_{13}n^{-1}. 
\]
This completes the proof.
\end{proof}

\subsection{Auxiliary Results}
Denote
\[
 L^+_{\tau}(y):= \frac{(2\theta)^m}{h^{m+1}} \sum_{j=0}^N C_{j,m} K^{(m)}\bigg(\frac{y-x_0-\theta(2j+m)}{h}\bigg).
\]
Then
\[
 {\rm var}_f[\hat{f}^+_\tau(x_0)] = \rE_f [\xi_\tau(x_0)]^2,\;\;\;\;\xi_\tau(x_0):=\frac{1}{n} \sum_{i=1}^n 
 \big[L^+_\tau(Y_i)- \rE_f L^+_\tau(Y_i)\big].
\]
\par 
Let 
\begin{eqnarray}
 \zeta(x_0) &:=& \sup_{\tau \in \cT} \big[|\xi_{\tau}(x_0)| - \Lambda_\tau(\kappa)\big]_+ 
 \label{eq:zeta-1}
 \\
 \eta(x_0) &:=& \sup_{\tau\in \cT} \big[ |\hat{\sigma}^2_\tau - \sigma^2_\tau| - u_\tau \Lambda_\tau(\kappa)\big]_+.
 \label{eq:zeta-2}
\end{eqnarray}

\begin{lemma}\label{lem:xi}
 For any $p\geq 1$ and $\kappa>0$ one has 
 \begin{eqnarray*}
  \rE_f [\zeta (x_0)]^p &\leq&  2\Gamma(p+1) M_h M_N  \big[\Lambda_\tau(\kappa)\big]^p \kappa^{-p} e^{-\kappa},
  \\
  \rE_f [\eta (x_0)]^p &\leq&  2\Gamma(p+1) M_h M_N  
  \big[u_\tau \Lambda_\tau(\kappa)\big]^p \kappa^{-p} e^{-\kappa}.
 \end{eqnarray*}
\end{lemma}
\begin{proof}
(i). Observe that  
 $|L_\tau^+(Y_j)| \leq u_\tau/2$, where $u_\tau$ is defined in (\ref{eq:eta-u}); hence   
 $|\xi_\tau|\leq u_\tau$. In addition, it follows from   (\ref{eq:var-1}) that 
 \[
  {\rm var}_f \big[L^+_\tau(Y_1)\big] \leq \sigma_\tau^2
  :=\frac{(2\theta)^{2m}}{h^{2m+2}}\sum_{j=0}^NC_{j,m}^2\int_{-\infty}^\infty
 \left|K^{(m)}\left(\frac{y-x_0-\theta(2j+m)}{h}\right)\right|^2f_Y(y) \rd y.
 \]
By Bernstein's inequality for any $z>0$
 \[
  \rP_f \big\{ |\xi_\tau(x_0)|\geq z\big\} \leq 2\exp\Big\{-\frac{nz^2}{2\sigma^2_\tau + \tfrac{2}{3} u_\tau z}
  \Big\}.
 \]
Therefore for $\Lambda_\tau(\kappa)$ defined in (\ref{eq:Lambda-tau})
 we obtain
 \begin{align}\label{eq:exp-ineq}
 \rP_f \big\{ |\xi_\tau(x_0)| \geq  \Lambda_\tau(\kappa) \big\} \leq 
 2\exp\bigg\{-  \frac{\big(\sigma_\tau\sqrt{\frac{2\kappa}{n}}  + 
 \frac{2}{3}  u_\tau \kappa n^{-1}\big)^2}
 {2\sigma_\tau^2/n + \frac{2u_\tau}{3n} \big(\sigma_\tau\sqrt{\frac{2\kappa}{n}} + 
 \frac{2\kappa u_\tau}{3n}\big)}
 \bigg\} \leq 2e^{-\kappa},
\end{align}
where we have used the following elementary inequality: for any $a>0, b>0$ and  $\kappa>0$ 
\begin{equation}\label{eq:elementary}
 \frac{ (\sqrt{\kappa} a+ \kappa b)^2}{a^2+ b(\sqrt{\kappa} a+ \kappa b)} \geq \kappa.
\end{equation}
Therefore, for any $p\geq 1$
\begin{align}
 \rE_f \big[ |\xi_\tau(x_0)|  &-  \Lambda_\tau(\kappa) \big]_+^p = p\int_0^\infty t^{p-1}
 \rP_f \big\{ |\xi_\tau(x_0)|  \geq   \Lambda_\tau(\kappa) +t \big\}\rd t
 \nonumber
 \\
 &\leq  p\big[\Lambda_\tau(\kappa)\big]^p \int_0^\infty y^{p-1} 
 \rP_f \big\{ |\xi_\tau(x_0)|  \geq   \Lambda_\tau(\kappa (1 +y) ) \big\} \rd y 
\nonumber
 \\
 &\leq  2p[\Lambda_\tau(\kappa)]^p \int_0^\infty y^{p-1} e^{-\kappa(1+y)} \rd y = 
 2 \Gamma(p+1) \big[\Lambda_\tau(\kappa)\big]^p \kappa^{-p} e^{-\kappa},
\label{eq:expectation-xi}
 \end{align}
where the second line follows from the change of variables and the fact that 
$\Lambda_\tau(a \kappa)\leq a \Lambda_\tau(\kappa)$ for $a\geq 1$; and the third line
is a consequence of (\ref{eq:exp-ineq}).
 \par 
(ii). Let $\hat{\sigma}^2_\tau$ be the emripical estimator for $\sigma^2_{\tau}$ based on the sample 
$Y_1,Y_2,\ldots,Y_n$: 
\begin{align*}
\hat{\sigma}_{\tau}^2 :=\frac{(2\theta)^{2m}}{nh^{2m+2}}\sum_{i=1}^n\sum_{j=0}^NC_{j,m}^2\left|K^{(m)}\left(\frac{Y_i-x_0-\theta(2j+m)}{h}\right)\right|^2.
\end{align*}
 Then
\begin{align*}
\hat{\sigma}_\tau^2-\sigma_\tau^2=\frac{1}{n}\sum_{i=1}^n 
\Big(\psi_{\tau}(Y_i)- \rE_f[\psi_{\tau}(Y_i)]\Big),
\end{align*}
where we put
\[
 \psi_\tau(y):= \frac{(2\theta)^{2m}}{h^{2m+2}} \sum_{j=0}^N C_{j,m}^2 
 \left|K^{(m)}\left(\frac{y-x_0-\theta(2j+m)}{h}\right)\right|^2.
\]
It is evident that 
\[
|\psi_\tau(y)|\leq \frac{(2\theta)^{2m}}{h^{2m+2}} C^2_{N,m} \|K^{(m)}\|_\infty^2 = \tfrac{1}{4} u_\tau^2,\;\;\;\forall y;
\]
hence $\big|\psi_{\tau}(Y_i)- \rE_f[\psi_{\tau}(Y_i)]\big|\leq u_\tau^2/4$, and 
\[
 {\rm var}_f\{\psi_\tau(Y_i)\} \leq \rE_f \big[\psi_\tau^2(Y_i)\big] \leq \tfrac{1}{4}\sigma_\tau^2 u^2_\tau.
\]
Therefore by Bernstein inequality  for any $z\geq 0$
\[
 \rP_f \Big\{ \big|\hat{\sigma}_\tau^2-\sigma^2_\tau\big|\geq z\Big\} \leq 
 2\exp\bigg\{- \frac{nz^2}{\tfrac{1}{2}\sigma_\tau^2 u_\tau^2 + \tfrac{1}{6} u_\tau^2 z}\bigg\}.
\]
This inequality together with (\ref{eq:elementary}) implies that 
\[
\rP_f\Big\{|\hat{\sigma}_\tau^2-\sigma_\tau^2|\geq u_\tau \Lambda_\tau(\kappa) \Big\} \leq 
 \rP_f\bigg\{|\hat{\sigma}_\tau^2- \sigma^2_\tau|\geq u_\tau\Big(\sigma_\tau \sqrt{\frac{\kappa}{2n}}+ 
 \frac{u_\tau \kappa}{6n}\Big)\bigg\} \leq  2 e^{-\kappa}.
\]
Similarly to the derivation in (\ref{eq:expectation-xi}) we have for any $p\geq 1$
\begin{align*}
\rE_f \big[|\hat{\sigma}^2_\tau &- \sigma^2| - u_\tau \Lambda_\tau(\kappa)\big]_+^p =
p\int_0^\infty t^{p-1}
 \rP_f \big\{ |\hat{\sigma}^2_\tau -\sigma^2_\tau|  \geq  u_\tau \Lambda_\tau(\kappa) +t \big\}\rd t
 \nonumber
 \\
 &\leq  p\big[u_\tau \Lambda_\tau(\kappa)\big]^p \int_0^\infty y^{p-1} 
 \rP_f \big\{ |\hat{\sigma}_\tau^2-\sigma_\tau^2|  \geq   u_\tau 
 \Lambda_\tau(\kappa (1 +y) ) \big\} \rd y 
\nonumber
\\
&\leq 
2p[u_\tau \Lambda_\tau(\kappa)]^p \int_0^\infty y^{p-1} e^{-\kappa(1+y)} \rd y = 
 2 \Gamma(p+1) \big[u_\tau \Lambda_\tau(\kappa)\big]^p \kappa^{-p} e^{-\kappa}.
\end{align*}
This completes the proof.
\end{proof}
\par 
Denote
\[
 \tilde{\Lambda}_\tau(\kappa) := \hat{\sigma}_\tau \sqrt{\frac{2\kappa}{n}} + \frac{2u_\tau \kappa}{3n}
\]
and observe that $\tilde{\Lambda}_\tau(\kappa)=\tfrac{1}{7}\hat{\Lambda}_\tau(\kappa)$, where 
$\hat{\Lambda}_\tau(\kappa)$ is defined in  (\ref{eq:Lambda-hat}).
\begin{lemma}\label{lem:Lambda-Lambda-tilde}
 For any $\tau\in \cT$ one has  
 \begin{equation}\label{eq:LL}
  \big[ \Lambda_\tau(\kappa)-7 \tilde{\Lambda}_\tau(\kappa)\big]_+ \leq 2c\eta(x_0),\;\;\;
  \big[\tilde{\Lambda}_\tau(\kappa) - 6 \Lambda_\tau(\kappa)\big]_+ \leq \eta (x_0) + \frac{\kappa}{n},
 \end{equation}
where $\eta(x_0)$ is defined in (\ref{eq:zeta-2}) and $c:=2^{-m-2}\theta \|K^{(m)}\|_\infty^{-1}$.
\end{lemma}
\begin{proof}
 We have 
 $\tilde{\Lambda}_\tau (\kappa) - \Lambda_\tau (\kappa)= (\hat{\sigma}_\tau - \sigma_\tau) \sqrt{2\kappa/n}$.
Define
 \[ 
 \cT_1:=\Big\{\tau\in \cT: \sigma_\tau\sqrt{\frac{2\kappa}{n}} \geq \frac{4u_\tau \kappa}{n}\Big\}.
\]
If  $\tau\in \cT_1$ then 
$\sigma_\tau\geq 2\sqrt{2} u_\tau (\kappa /n)^{1/2}$ and 
\begin{align*}
 |\hat{\sigma}_\tau - \sigma_\tau| = \frac{|\hat{\sigma}^2_\tau - \sigma^2_\tau|}{\hat{\sigma}_\tau+\sigma_\tau}
\leq \frac{1}{\sigma_\tau} |\hat{\sigma}^2_\tau - \sigma^2_\tau|
\leq \frac{1}{2 u_\tau} \sqrt{\frac{n}{2\kappa}} \big[\eta(x_0) + u_\tau \Lambda_\tau(\kappa)\big];
 \end{align*}
hence for any $\tau\in \cT_1$
\begin{equation}\label{eq:Lambda-tilde-1}
 |\tilde{\Lambda}_\tau(\kappa) - \Lambda_\tau(\kappa)| \leq \tfrac{1}{2}\Lambda_\tau (\kappa) + 
 \frac{\eta(x_0)}{2u_\tau}  \leq \tfrac{1}{2}\Lambda_\tau(\kappa)+ c \eta(x_0) 
\end{equation}
where we have used that \mbox{$u_\tau\geq 2^{m+1} \theta^{-1}\|K^{(m)}\|_\infty$} for all 
\mbox{$\tau\in \cT$}, and 
denoted for brevity \mbox{$c:= 2^{-m-2}\theta \|K^{(m)}\|_\infty^{-1}$}.
Thus (\ref{eq:Lambda-tilde-1}) implies that 
\begin{equation}\label{eq:111}
 \big[ \tilde{\Lambda}_\tau(\kappa) - \tfrac{3}{2}\Lambda_\tau(\kappa)\big]_+ \leq   c\eta(x_0)\;\;\;
 \hbox{and}\;\;\; \big[ \Lambda_\tau(\kappa) - 2\tilde{\Lambda}_\tau(\kappa)\big]_+ 
 \leq  2c\eta(x_0),
 \;\;
 \forall \tau\in \cT_1.
\end{equation}
\par 
Now assume that $\tau\in \cT_2:=\cT \setminus \cT_1$; 
for such $\tau$, $\Lambda_\tau(\kappa) \leq \frac{14}{3}u_\tau\kappa/n$.
Note also that by definition $\tilde{\Lambda}_\tau(\kappa) \geq \frac{2}{3} u_\tau \kappa/n$; therefore
\begin{equation}\label{eq:222}
 [\Lambda_\tau(\kappa)- 7\tilde{\Lambda}_\tau(\kappa)]_+ =0,\;\;\;\forall \tau\in \cT_2.
\end{equation}
Furthermore, 
we bound $|\hat{\sigma}_\tau-\sigma_\tau|$ as follows:
\begin{align*}
 |\hat{\sigma}_\tau - \sigma_\tau| \leq |\hat{\sigma}^2_\tau - \sigma_\tau^2|^{1/2}  
\leq 
 \sqrt{\eta(x_0)}
 + \sqrt{u_\tau \Lambda_\tau(\kappa)} \leq \sqrt{\eta(x_0)} + \sqrt{5} u_\tau \sqrt{\frac{\kappa}{n}}.
\end{align*}
Therefore 
for any $\tau\in\cT_2$
\begin{equation*}
 \big|\tilde{\Lambda}_\tau(\kappa) - \Lambda_\tau(\kappa)\big| \leq \sqrt{\frac{2\kappa}{n} \eta(x_0)}
 + \sqrt{10}\, \frac{u_\tau\kappa}{n} \leq \frac{\kappa}{n}+ \eta(x_0) + 5\Lambda_\tau(\kappa),
\end{equation*}
where the last bound follows from the  elementary inequality $\sqrt{2ab}\leq \sqrt{a^2+b^2}\leq a+b$
for $a, b\geq 0$.
This implies that 
\begin{equation}\label{eq:333}
 \big[\tilde{\Lambda}_\tau(\kappa) - 6 \Lambda_\tau(\kappa)\big]_+ \leq  \frac{\kappa}{n} + \eta(x_0),
 \;\;\;\forall \tau\in \cT_2.
\end{equation}
Combining (\ref{eq:111}), (\ref{eq:222}) and (\ref{eq:333}) we complete the proof.
\end{proof}

\end{appendix}
\end{document}